\documentclass[11pt,reqno]{amsart}
\usepackage{bm,amsmath,amsthm,amssymb,mathtools,verbatim,amsfonts,tikz-cd,mathrsfs,diagbox,makecell}
\usepackage{enumitem}
\usepackage{float}

\usepackage[hidelinks,colorlinks=true,linkcolor=blue, citecolor=black,linktocpage=true]{hyperref}
\usepackage{graphicx}
\usepackage[alphabetic]{amsrefs}
\usepackage{caption}
\usepackage{morefloats}
\usepackage[top=1.4in, bottom=1.2in, left=1.4in, right=1.4in,marginpar=1in]{geometry}
\usepackage{subfigure}
\usepackage{adjustbox}

\usetikzlibrary{matrix,arrows,decorations.pathmorphing}

\usepackage{chngcntr}
\counterwithin{figure}{section}

\newtheorem{thm}{Theorem}[section]
\newtheorem{prop}[thm]{Proposition}

\newtheorem{lem}[thm]{Lemma}

\theoremstyle{definition}
\newtheorem{define}[thm]{Definition}

\theoremstyle{remark}
\newtheorem{rem}[thm]{Remark}
\newtheorem{example}[thm]{Example}

\newcommand{\ve}[1]{\boldsymbol{\mathbf{#1}}}

\newcommand{\Z}{\mathbb{Z}}

\renewcommand{\d}{\partial}
\renewcommand{\subset}{\subseteq}

\renewcommand{\tilde}{\widetilde}

\newcommand{\iso}{\cong}

\DeclareMathOperator{\End}{{End}}

\DeclareMathOperator{\Ext}{{Ext}}

\DeclareMathOperator{\gr}{{gr}}

\DeclareMathOperator{\Hom}{{Hom}}

\DeclareMathOperator{\id}{{id}}

\DeclareMathOperator{\Int}{{int}}

\DeclareMathOperator{\rank}{{rank}}

\DeclareMathOperator{\Span}{{Span}}

\newcommand{\lk}{\mathrm{lk}}

\newcommand{\bF}{\mathbb{F}}

\newcommand{\bH}{\mathbb{H}}
\newcommand{\bI}{\mathbb{I}}

\newcommand{\cA}{\mathcal{A}}
\newcommand{\cB}{\mathcal{B}}
\newcommand{\cC}{\mathcal{C}}

\newcommand{\cH}{\mathcal{H}}

\newcommand{\cK}{\mathcal{K}}

\newcommand{\cR}{\mathcal{R}}
\newcommand{\cS}{\mathcal{S}}
\newcommand{\cT}{\mathcal{T}}

\newcommand{\cX}{\mathcal{X}}
\newcommand{\cY}{\mathcal{Y}}

\newcommand{\fro}{\mathfrak{o}}

\newcommand{\scA}{\mathscr{A}}

\newcommand{\cCFL}{\mathcal{C\!F\!L}}
\newcommand{\cCFK}{\mathcal{C\hspace{-.5mm}F\hspace{-.3mm}K}}
\newcommand{\cHFL}{\mathcal{H\!F\! L}}
\newcommand{\cHFK}{\mathcal{H\!F\! K}}

\newcommand{\HF}{\mathit{HF}}

\newcommand{\xs}{\ve{x}}
\newcommand{\ys}{\ve{y}}
\newcommand{\zs}{\ve{z}}
\newcommand{\ws}{\ve{w}}

\newcommand{\ps}{\ve{p}}

\newcommand{\Wh}{\mathrm{Wh}}

\renewcommand{\a}{\alpha}
\renewcommand{\b}{\beta}

\newcommand{\veps}{\varepsilon}

\usepackage{leftidx}

\DeclareMathOperator{\Cone}{{Cone}}

\numberwithin{equation}{section}

\newcommand{\ar}{\mathrm{a.r.}}

\newcommand{\llsquare}{[\hspace{-.5mm}[}
\newcommand{\rrsquare}{]\hspace{-.5mm}]}

\newcommand{\alg}{\mathrm{alg}}

\allowdisplaybreaks

\newcommand{\cCo}{\cC\! o}
\newcommand{\cTr}{\cT\! r}

\DeclareMathOperator{\opp}{opp}

\DeclareMathOperator{\Tw}{Tw}

\DeclareMathOperator{\wt}{wt}

\newcommand{\MOD}{\mathsf{Mod}}

\title{The link surgery modules of 2-component L-space links}

\author{Daren Chen}
\address{Department of Mathematics\\California Institute of Technology\\ Pasadena, CA, USA}
\email{darenc@caltech.edu}

\author{Ian Zemke}
\address{Department of Mathematics\\University of Oregon\\  Eugene, OR, USA}
\email{izemke@uoregon.edu}

\author{Hugo Zhou}
\address{Department of Mathematics\\University of Michigan\\  Ann Arbor, MI, USA}
\email{hugozhou@umich.edu}

\begin{document}
\maketitle
\begin{abstract} 
In our earlier work \cite{CZZSatellites}, we studied the link surgery modules of two component L-space links. Therein, we computed two of the four idempotents of such modules.  In this article, we use Koszul duality to give an alternate account of this proof, and also to extend it to compute the entire link surgery modules of such links, modulo a technical result which will be proven in a subsequent paper.
\end{abstract}

\tableofcontents

\section{Introduction}

If $L$ is a link in $S^3$ with integral framing $\Lambda$, the \emph{link surgery complex} of $L$, denoted $\cC_{\Lambda}(L)$, is a chain complex defined by Manolescu and Ozsv\'{a}th \cite{MOIntegerSurgery} which computes the Heegaard Floer homology of $S^3_{\Lambda}(L)$. 

The \emph{link Floer complex} is an invariant of links, defined by Ozsv\'{a}th and Szab\'{o} \cite{OSLinks}. It takes the form of a free and finitely generated chain complex $\cCFL(L)$ over the ring $\bF[W_1,Z_1,\dots, W_n,Z_n]$, where $n=|L|$.  Conceptually, we can think of the link surgery complex $\cC_{\Lambda}(L)$ is an enhancement of the complex $\cCFL(L)$ which incorporates some additional algebraic data necessary for the surgery formula.

In \cite{ZemBordered,ZemExact}, the second author reinterpreted the link surgery complex as an invariant for 3-manifolds with parametrized torus boundary components (i.e. the link complement), and repackaged the surgery complex in terms of certain $A_\infty$-module categories over an algebra, called the \emph{surgery algebra}, denoted $\cK$. This theory was modeled on the bordered theory of Lipshitz, Ozsv\'{a}th and Thurston \cite{LOTBordered}, and features a gluing theorem for the invariant of a manifold obtained by gluing two 3-manifolds with torus boundaries together.

To a link with 2-components $L\subset S^3$, the second author defines in \cite{ZemBordered} an $A_\infty$-bimodule (more precisely, a $DA$-bimodule), denoted ${}_{\cK} \cX(L)^{\cK}$,
which encodes the link surgery complex of $L$. The algebraic structure of this invariant as a bimodule (as opposed to just a chain complex) allows for the gluing formulas from \cite{ZemBordered}.

In this article, we study the link surgery complexes and bimodules of a family of 2-component links called \emph{L-space links}. We recall that a link $L\subset S^3$ is an \emph{L-space link} if $S^3_{\Lambda}(L)$ is an L-space for all integral surgeries $\Lambda \gg 0$. (Recall that a rational homology 3-sphere is an \emph{L-space} if $\rank_{\Z/2} \widehat{\HF}(Y)$ is equal to the number of elements in $H_1(Y)$).

L-space knots and links are well-studied in the literature. Ozsv\'{a}th and Szab\'{o} proved in \cite{OSlens} that the knot Floer complex of an L-space knot $K$ is a staircase complex. For example, the complex of the $(3,4)$ torus knot is freely generated over $\bF[W,Z]$ by five elements, $\xs_0$, $\xs_1$, $\xs_2$, $\xs_3$ and $\xs_4$, and has differential shown below:
\begin{equation}
\begin{tikzcd}[labels=description, column sep=.5cm]& \xs_1 \ar[dl,"W"] \ar[dr, "Z^2"]&& \xs_3 \ar[dl, "W^2"] \ar[dr, "Z"]\\
	\xs_0 && \xs_2&& \xs_4
\end{tikzcd}.
\label{eq:T(3,4)}
\end{equation}
In \cite{CZZSatellites}, we extended this computation to 2-component L-space links. We proved that if $L$ is a 2-component L-space link, then $\cCFL(L)$ is formal, i.e., homotopy equivalent to a free-resolution of its homology. 

Work of the second author with Borodzik and Liu \cite{BLZLattice} extends Ozsv\'{a}th and Szab\'{o}'s result in a different direction. Therein, it is proven that if $L$ is a \emph{plumbed} L-space link in $S^3$, then $\cCFL(L)$ is also formal. It is an open question whether non-plumbed L-space links $L$ with more than 2-components have formal link Floer complexes.

Using work of Gorsky and N\'{e}methi \cite{GorskyNemethiLattice}, Liu \cite{Liu-L-spaces} describes a reduced model of the link surgery complex $\cC_{\Lambda}(L)$ of a 2-component L-space link. This computation is sufficient for computing the Heegaard Floer homologies of surgeries on 2-component L-space links, but it is not sufficient for understanding the bimodule structure on ${}_{\cK} \cX(L)^{\cK}$, which is necessary for the gluing formulas from \cite{ZemBordered}. 

For the theorem statements in this paper, it is important also that the link surgery bimodule ${}_{\cK} \cX(L)^{\cK}$ depends on an extra choice of data, called an \emph{arc-system}. For the purpose of this introduction, an arc system is equivalent to a designation of each component of $L$ as having either an ``alpha parallel arc'' or a ``beta parallel arc''. (The meaning of this designation is not important for our paper).

In this paper, we study the $DA$-bimodules of 2-component L-space links over the surgery algebra. In \cite{CZZSatellites}, we gave a partial computation of these bimodules. We computed two of the four idempotents of the module ${}_{\cK} \cX(L)^{\cK}$. That is, inside of $\cK$ there is a subalgebra $R_0:=\bF[W,Z]\subset \cK$, and we computed the corresponding bimodule with outputs restricted to this subalgebra. The main goal of this paper is to compute the entire complex (modulo some technical arguments that will appear in the subsequent \cite{ZemUEq}).

Before stating our theorem, we need the following technical definition:

\begin{define}
 We say that the structure maps $\delta_{n+1}^1$ of a  type-$DA$ bimodule ${}_{\cK} \cX^{\cK}$ are \emph{$U$-equivariant} if
\[
\delta_{n+1}^1(a_n,\dots,U a_i,\dots a_1,\xs)=\delta_{n+1}^1(a_n,\dots, a_i,\dots, a_1,\xs)\cdot U
\]
for all $a_n,\dots, a_1\in \cK$ and $\xs\in \cX$. Here, $\cdot U$ means to multiply the $\cK$ factor of the output of $\delta_{n+1}^1$ by $U$.
\end{define}

In this article, we define a candidate bimodule ${}_{\cK} \cX^{\alg}(L)^{\cK}$ for a 2-component L-space link $L\subset S^3$. This bimodule is determined by the $H$-function of $L$. By work of Gorsky and N\'{e}methi \cite{GorskyNemethiLattice}, the $H$-function is determined by the Alexander polynomials of $L$ and its sublinks, so we can also view $\cX^{\alg}(L)$ as being determined by the Alexander polynomials of $L$ and its sublinks.

\begin{thm}
\label{thm:main-computation-intro} Let $L\subset S^3$ be a two component L-space link. If $\scA$ is a choice of arc-system for $L$ such that the bimodule ${}_{\cK} \cX(L;\scA)^{\cK}$ has a model with $U$-equivariant structure maps, then there is a homotopy equivalence
\[
{}_{\cK} \cX(L;\scA)^{\cK} \simeq {}_{\cK} \cX^{\alg}(L)^{\cK}. 
\]
\end{thm}

In a forthcoming article \cite{ZemUEq}, the second author proves the following:

\begin{thm}[\cite{ZemUEq}]
\label{thm:forthcoming}
 If $L\subset S^3$ is a 2-component link and $\scA$ is an arc-system where the arc for one component of $L$ is beta-parallel while the arc for the other component is alpha-parallel, then ${}_{\cK} \cX(L;\scA)^{\cK}$ admits a  $U$-equivariant model.
\end{thm}

In particular, Theorem~\ref{thm:forthcoming} implies that there is an arc-system for $L$ for which Theorem~\ref{thm:main-computation-intro} applies. 

\begin{rem}  There are arc systems on L-space links for which the structure maps on ${}_{\cK} \cX(L;\scA)^{\cK}$ cannot be made $U$-equivariant. For example, in \cite{ZemBordered}*{Section~16,17}, the second author computed a reduced model of the bimodule ${}_{\cK} \cX(L;\scA)^{\cK}$ when $L$ was a Hopf link and $\scA$ consisted of two alpha-parallel arcs, and the resulting structure maps are not $U$-equivariant. In fact, it is not hard to see that the this bimodule is not even homotopy equivalent to a $U$-equivariant bimodule. 
\end{rem}

\subsection*{Acknowledgments}

The first author was partially supported by the Simons Collaboration Grant on New Structures in Low Dimensional Topology. The second author was partially supported by NSF grant 251324 and a Sloan fellowship. The third author was partially supported by an AMS-Simons travel grant.

\section{Algebraic background}

\subsection{Type-$D$, $A$, $DA$ and $DD$ bimodules}

In this section, we recall the algebraic framework of type-$D$, $A$ and $DA$ bimodules due to Lipshitz, Ozsv\'{a}th and Thurston, as well as the natural analogs of these notions for bimodules over curved $dg$-algebras. We refer the reader to 
 \cite{LOTBordered} \cite{LOTBimodules} for further details on the uncurved setting, and we refer the reader to \cite{OSBordered=HFK}*{Section~3} for additional background on this framework in the presence of curvature. 

Let $\ve{k}$ be a unital ring with characteristic 2. A \emph{curved $dg$-algebra} $(\cA,\mu_i)$ over $\ve{k}$ consists of a $(\ve{k},\ve{k})$-bimodule $\cA$, equipped with $\ve{k}$-linear operations 
\[
\mu_2\colon \cA\otimes_{\ve{k}} \cA\to \cA, \quad \mu_1\colon \cA\to \cA,\quad \mu_0\colon \ve{k}\to \cA
\]
satisfying the following:
\begin{enumerate}
\item $\cA$ is an associative algebra with respect to $\mu_2$.
\item $\mu_1^2=0$, and furthermore $\mu_2$ satisfies the Leibniz rule with respect to $\mu_1$: $\mu_1(a\cdot b)=\mu_1(a)\cdot b+a\cdot \mu_1(b).$ (Here $a\cdot b$ denotes $\mu_2(a,b)$).
\item The element $\mu_0:=\mu_0(1)$ is a cycle and is central, i.e.
\[
\mu_0\cdot a=a\cdot \mu_0\quad \text{and} \quad \mu_1(\mu_0)=0,
\]
for all $a\in \cA$.
\end{enumerate}

If $(\cA,\mu_i)$ is a curved $dg$-algebra, a right \emph{type-$D$ module} over $\cA$, denoted $\cX^\cA$, consists of a right $\ve{k}$-module $\cX$ equipped with $\ve{k}$-linear map
\[
\delta^1\colon \cX\to \cX\otimes \cA
\]
which satisfies
\[
(\bI_{\cX}\otimes \mu_2)\circ(\delta^1\otimes \bI_{\cA})\circ \delta^1+(\bI_{\cX}\otimes \mu_1)\circ \delta^1+\bI_{\cX}\otimes \mu_0=0.
\]

If $(\cA,\mu_i)$ and $(\cB,\nu_i)$ are two curved $dg$-algebras over $\ve{k}$ and $\ve{j}$, respectively, then a $DA$-bimodule ${}_{\cA} \cX^{\cB}$ over $(\cA,\cB)$ consists of a  $(\ve{k},\ve{j})$-bimodule $\cX$, equipped with a collection of $(\ve{k},\ve{j})$-linear structure maps 
\[
\delta_{j+1}^1\colon \underbrace{\cA\otimes_{\ve{k}} \cdots \otimes_{\ve{k}} \cA}_j\otimes \cX\to \cX\otimes_{\ve{j}} \cB
\]
which satisfy the following compatibility condition. We have
\[
(\bI_{\cX}\otimes \mu_2)\circ (\delta_1^1\otimes \bI_{\cA})\circ \delta_1^1+(\bI_{\cX}\otimes \mu_1)\circ \delta_1^1+\bI\otimes \nu_0=0.
\]
Furthermore,  for each $n\ge 1$ and each $a_n,\dots, a_1\in \cA$ and $\xs\in \cX$, we have
\[
\begin{split}
0=&\sum (\bI_{\cX}\otimes \nu_2)(\delta_{n-j+1}^1\otimes \bI_{\cB})(a_n,\dots,a_{j+1},\delta_{j+1}^1(a_j,\dots, a_1,\xs))\\
 +&\sum \delta_{n-i+2}^1(a_n,\dots, \mu_i(a_{j+i-1},\dots, a_j),\dots, a_1,\xs)\\
 +&(\bI_{\cX}\otimes \nu_1)\left(\delta_{n+1}^1(a_n,\dots, a_1,\xs)\right).
\end{split} 
\]
Finally, we discuss $DD$-bimodules. If $\cA$ and $\cB$ are curved $dg$-algebras over $\ve{j}$ and $\ve{k}$, respectively, a $DD$-bimodule ${}^{\cA} \cX^{\cB}$ consists of a $(\ve{j},\ve{k})$-bimodule $\cX$, equipped with a structure map 
\[
\delta^{1,1}\colon \cX\to \cA\otimes_{\ve{j}} \cX\otimes_{\ve{k}} \cB.
\]
The compatibility condition is the same as if  we view $\delta^{1,1}$ as a map from $\cX$ to $\cX\otimes_{\ve{j}^{\opp}\otimes \ve{k}} (\cA^{\opp}\otimes_{\bF} \cB)$, and view $\cX$ as a type-$D$ structure over $\cA^{\opp}\otimes_{\bF} \cB$. Here, $\cA^{\opp}$ denotes the curved $dg$ algebra with underlying space $\cA$, but with $\mu_2^{\opp}(a_1,a_2)=\mu_2(a_2,a_1)$ (and the same differential and curvature). The tensor product $\cA^{\opp}\otimes_{\bF} \cB$ is equipped with the curvature $\mu_0\otimes 1+1\otimes \mu_0$, and with the differential $\mu_1\otimes \id+\id\otimes \mu_1$.

\subsection{A perturbation lemma}

In this section we prove a simple perturbation lemma, which will be useful later. We recall that a \emph{$dg$-category} is a category where the morphism sets are chain complexes of abelian groups, and whose differential satisfies the Leibniz rule with respect to composition. For our purposes, we will only consider the case where the morphism sets have characteristic 2, so that the differential satisfies
\[
\d(g\circ f)=\d(g)\circ f+g\circ \d(f).
\]

If $\cC$ is a $dg$-category, a \emph{twisted complex} in $\cC$ consists of a pair $\Tw_{\a}(X):=(X,\a)$ where $X$ is an object of $\cC$ and $\a\in \Hom_{\cC}(X,X)$ is a morphism satisfying the Mauer-Cartan relation:
\[
\d(\a)+\a\circ \a=0.
\]

The set of twisted complexes forms a category.
If $\Tw_{\a}(X)$ and $\Tw_\b(Y)$ are twisted complexes, we define
\[
\Hom_{\Tw(\cC)}(\Tw_{\a}(X),\Tw_\b(Y)):=\Hom_{\cC}(X,Y),
\]
with differential
\[
\d_{\Tw(\cC)}(f)=\d_{\cC}(f)+\b\circ f+f\circ \a.
\]
It is easy to check that $\Tw(\cC)$ is a $dg$-category. (Compare \cite{SeidelFukaya}*{Section~1.5}).

The categories of interest to us (chain complexes, type-$D$ modules, $A_\infty$-modules, etc.) are all closed under twists, in the sense that there is a canonical functor from $\Tw(\cC)$ to $\cC$. For example, in the category of chain complexes, this functor sends $\Tw_{\a}(X,d)$ to $(X,d+\a)$. Similarly for the $dg$-categories consisting of type-$D$ and type-$DA$ bimodules over $dg$-algebras, the twisted complex $\Tw_\a((X,\delta_{j+1}^1))$ is sent to $(X,\a_{j+1}^1+\delta_{j+1}^1)$. Because of this, we typically conflate $\Tw_\a(X)\in \Tw(\cC)$ and its image in $\cC$.

\begin{lem}
\label{lem:non-standard-perturbation-lemma} Suppose that $X$ is an object of a $dg$ category and $h\in \End(X)$ is an endomorphism such that $1+h$ is invertible. Then $\a=\d(h)(1+h)^{-1}$ is a Mauer-Cartan element and $X$ is isomorphic to $\Tw_{\a}(X)$.
\end{lem}
\begin{proof} To see that $\a$ is a Mauer-Cartan element, we observe that
\[
\d(\a)+\a^2= \d(h) \d( (1+h)^{-1})+\d(h)(1+h)^{-1}\d(h)(1+h)^{-1}.
\]
This is zero if and only if
\[
\d(h) \d( (1+h)^{-1}) (1+h)+\d(h)(1+h)^{-1}\d(h)
\]
is zero. We note that $\d((1+h)^{-1}) (1+h)=(1+h)^{-1}\d(1+h)=(1+h)^{-1}\d(h)$, and hence the above equation vanishes.

We now define a morphism $f$ from $X$ to $\Tw_{\a}(X)$ to be $1+h$. By definition, the twisted morphism differential of $f$ is equal to
\[
\d(f)+\a\circ f,
\]
(where $\d(f)$ is the differential of $f$ when viewed as an endomorphism of $X$). With our choice of $\a$, we observe that the above is equal to
\[
\d(h)+\d(h) (1+h)^{-1} (1+h)=0,
\]
so $f$ is a cycle with respect to the differential on morphisms of twisted complexes.  Similarly, we can define a morphism $g=(1+h)^{-1}$ from $\Tw_\a(X)$ to $X$ and it is straightforward to verify that $g$ is a cycle as well. We observe that $f$ and $g$ are inverses, so the claim is proven.
\end{proof}

\begin{rem} Our typical usage of the above lemma is as follows. We will consider a $DD$-bimodule ${}^{\cA} \cX^{\cB}$ with differential $\delta^{1,1}$. We will want to perturb the structure map $\delta^{1,1}$ by adding a term $\d(h)$ to it in order to simplify some terms of the differential. In general, this will not be possible because $\d(h)$ is not typically a Mauer-Cartan element. However, $\d(h)(1+h)^{-1}$ is a Mauer-Cartan element (as demonstrated in the previous lemma). Furthermore, in the cases of interest to us, the infinite sum $(1+h)^{-1}=\sum_{n=0}^\infty h^n$ makes sense, and furthermore we view $\a=\sum_{n=0}^\infty \d(h) h^n$ as coinciding with $\d(h)$ up to first approximation. 
\end{rem}

\section{Heegaard Floer background}

\subsection{Knot and link Floer homology}

We recall some basics about knot and link Floer homology. Knot Floer homology is an invariant of knots due independently to Ozsv\'{a}th and Szab\'{o} \cite{OSKnots} and Rasmussen \cite{RasmussenKnots}. Link Floer homology is an invariant of links due to Ozsv\'{a}th and Szab\'{o}.

 To a knot $K\subset Y$, the invariant takes the form of a free and finitely generated chain complex $\cCFK(Y,K)$ over $\bF[W,Z]$, where $\bF=\Z/2$.  We are primarily interested in the case that $Y=S^3$, in which case we write $\cCFK(K)$ for the invariant. There are three gradings of interest, denoted $\gr_{\ws},\gr_{\zs}$ and $A$, which are related by the equation
 \[
 A=\frac{1}{2}(\gr_{\ws}-\gr_{\zs}).
 \]
 The algebra $\bF[W,Z]$ is graded by
\[
(\gr_{\ws},\gr_{\zs},A)(W)=(-2,0,-1),\quad \text{and} \quad (\gr_{\ws},\gr_{\zs},A)(Z)=(0,-2,1). 
\]

It is convenient to write
\[
U:=WZ.
\]

If $L\subset S^3$ is a link, the invariant $\cCFL(L)$ takes the form of a free, finitely generated chain complex over the ring $\bF[W_1,\dots, W_n, Z_1,\dots, Z_n]$. This invariant has a more complicated set of gradings. We write 
\[
\bH(L)=\prod_{i=1}^n \left(\Z+\frac{\lk(K_i,L\setminus K_i)}{2}\right).
\]
There is an Alexander grading $A=(A_1,\dots, A_n)$ which takes values in $\bH(L)$. The variables have gradings
\[
A_i(Z_j)=-A_i(W_j)=\delta_{ij},
\]
where $\delta_{ij}$ is the Kronecker delta.

Similar to the case of knots, it is convenient to write $U_i=W_iZ_i$. All of the $U_i$ induce chain homotopic actions on homology. See, for example, \cite{BLZ_Non_Cuspidal}*{Equation~2.8}.

As before, there are Maslov gradings $\gr_{\ws}$ and $\gr_{\zs}$, though in fact there is a more refined collection of Maslov gradings. If $\fro$ is any orientation on $L$, then there is a Maslov grading
$\gr_{\fro}$ on $\cCFL(L)$. The construction of $\cCFL(L)$ requires a choice of base orientation, which we will denote $\fro_0$. The Maslov grading $\gr_{\fro_0}$ coincides with $\gr_{\ws}$, and more generally $\gr_{\fro}$ satisfies
\[
\gr_{\fro}(W_i)=
\begin{cases}-2 & \text{ if } \fro|_{K_i}=\fro_0|_{K_i},\\
0 & \text{ if } \fro|_{K_i}=-\fro_0|_{K_i} 
\end{cases}.
\]
Similarly
\[
\gr_{\fro}(Z_i)=
\begin{cases}0 & \text{ if } \fro|_{K_i}=\fro_0|_{K_i},\\
-2 & \text{ if } \fro|_{K_i}=-\fro_0|_{K_i} 
\end{cases}.
\]
We have the following relation:
\[
\gr_{\fro_0}-\gr_{\fro}=\left(\sum_{\fro|_{K_i}=-\fro_0|_{K_i}} 2A_i\right)-\lk(L_\fro,L\setminus L_\fro),
\]
where $L_\fro$ is the set of components of $L$ where $\fro$ differs from $\fro_0$. In computing the above linking number, we orient $L$ as in $\fro_0$. See \cite{ZemBordered}*{Lemma~6.4}.

\subsection{L-space knots and links}
\label{sec:L-space-links-background}

An important class of knots and links are \emph{L-space links}. A link $L$ is called an \emph{L-space link} if $S^3_{\Lambda}(L)$ is a Heegaard Floer L-space for all integral framings $\Lambda \gg 0$. (Recall that a rational homology 3-sphere is an \emph{L-space} if $\HF^-(Y)$ is isomorphic to $\bigoplus_{|H_1(Y)|} \bF[U]$, where $|H_1(Y)|$ is the number of elements in $H_1(Y)$). Using the large surgeries formula of Manolescu and Ozsv\'{a}th \cite{MOIntegerSurgery}*{Theorem~12.1}, an equivalent condition for $L$ to be an L-space link is that the module $\cHFL(L)$
is free as an $\bF[U]$ module, where $U$ acts by any of the $U_i=W_iZ_i$ (all $U_i$ have the same action on homology). Equivalently, $L$ is an L-space link if and only
\[
\cHFL(L)\iso \bigoplus_{\ve{s}\in \bH(L)} \bF[U]. 
\]

When $K$ is an L-space knot, the knot Floer homology group $\cHFL(L)$ has rank 0 or rank 1 in each $(\gr_{\ws},\gr_{\zs})$-grading. Furthermore, it is not hard to see that $\gr_{\ws}(\xs)\le 0$ and $\gr_{\zs}(\xs)\le 0$ for any non-zero, homogeneously graded element $\xs\in \cHFL(L)$. Therefore, $\cHFL(L)$ can be naturally embedded inside of $\bF[W,Z]$ as a monomial ideal (an ideal spanned by monomials). We illustrate the link Floer homology of the $(3,4)$-torus knot, $\cHFL(T_{3,4})$, in Figure~\ref{fig:monomial ideal of T(3,4)}.

\begin{figure}[h]
		\[
		\begin{tikzcd}[labels=description, column sep=.5cm]& \xs_1 \ar[dl,"W"] \ar[dr, "Z^2"]&& \xs_3 \ar[dl, "W^2"] \ar[dr, "Z"]\\
					\xs_0 && \xs_2&& \xs_4
				\end{tikzcd}
		 \raisebox{-1cm}{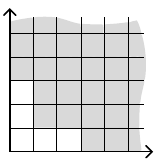}
		 \]
	\caption{The link Floer complex $\cCFK(T_{3,4})$ (left) and its homology $\cHFK(T_{3,4})$ (right). Each shaded square on the right denotes a generator over $\bF$. Multiplication by $W$ shifts to the right, while multiplication by $Z$ shifts upwards. We can view the entire upper-right quadrant as $\bF[W,Z]$.}
	\label{fig:monomial ideal of T(3,4)}
\end{figure}
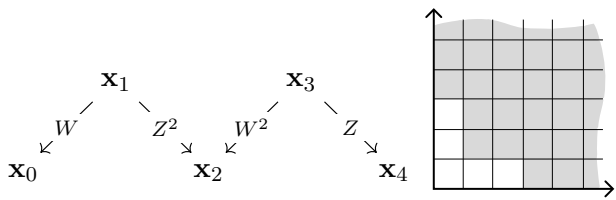

\subsection{The link surgery complex}

To a link $L\subset S^3$ with integral framing $\Lambda$, Manolescu and Ozsv\'{a}th \cite{MOIntegerSurgery} describe a chain complex $\cC_{\Lambda}(L)$. The chain complex is a module over $\bF\llsquare U_1,\dots, U_n\rrsquare$, where $n=|L|$. The homology  of $\cC_{\Lambda}(L)$ is isomorphic to the completion of the Heegaard Floer homology with respect to $(U)$:
\[
H_*(\cC_{\Lambda}(L))\iso \ve{\HF}^-(S^3_{\Lambda}(L)):=\HF^-(S^3_{\Lambda})\otimes_{\bF[U]} \bF\llsquare U\rrsquare.
\]
Work of the second author \cite{ZemExact} extends this theory to arbitrary Morse framed links in 3-manifolds.

We now sketch some properties of the link surgery complex, focusing on links in $S^3$. We refer the reader to \cite{MOIntegerSurgery} or \cite{ZemExact} for further details.  The link surgery complex $\cC_{\Lambda}(L)$ is a hypercube of chain complexes (see \cite{MOIntegerSurgery}*{Section~5} for background on this algebraic formalism). In more detail, the underlying group of $\cC_{\Lambda}(L)$ decomposes as a direct sum 
\[
\cC_{\Lambda}(L)\iso\bigoplus_{\veps\in \{0,1\}^n} \cC_{\veps}.
\]
Furthermore, the differential $D$ on $\cC_{\Lambda}(L)$ decomposes as a sum
\[
D=\sum_{ \veps\le \veps'} D_{\veps,\veps'},
\]
where $D_{\veps,\veps'}$ maps $\cC_{\veps}$ to $\cC_{\veps'}$.

The complex $\cC_{\vec{0}}$ is isomorphic to a completion of $\cCFL(L)$. More generally, there is a homotopy equivalence
\begin{equation}
\cC_\veps\simeq \cCFL(L_{\veps},\ve{p}_{\veps})\otimes_{\bF} \bF[T_{i_1},T_{i_1}^{-1},\dots, T_{i_k},T_{i_k}^{-1}].
\label{eq:C_epsilon=CFL-L-veps}
\end{equation}
In the above, $L_{\veps}\subset L$ is the sublink of components $K_i\subset L$ where $\veps_i=0$, and $\{i_1,\dots, t_k\}=\veps^{-1}(1)$. Also $\ve{p}_{\veps}$ is a collection of $n-|L_{\veps}|$ extra basepoints, which are disjoint from $L_{\veps}$.

The differential $D_{\veps,\veps'}$ decomposes further as a sum. Write 
\[
M_{\veps,\veps'}:=L_{\veps}\setminus L_{\veps'}.
\] 
The components of $M_{\veps,\veps'}$ are in one-to-one correspondence with the set of indices which are increased in $\veps'$ from $\veps$. The differential decomposes over the set of orientations of $M_{\veps,\veps'}$, for which we write $\fro(M_{\veps,\veps'})$. If $\fro$ is an orientation, write $M_{\veps,\veps'}^{\fro}$ for $M_{\veps,\veps'}$, equipped with this orientation. Following standard conventions, we write $\vec{M}_{\veps,\veps'}$ for a sublink of $L$, equipped with a choice of orientation. The map $D_{\veps,\veps'}$ decomposes as a sum
\[
D_{\veps,\veps'}=\sum_{\fro(M_{\veps,\veps'})} \Phi^{M^{\fro}_{\veps,\veps'}}_{\veps,\veps'},
\]
where the sum is over the set of orientations on $M$. We frequently omit the subscripts and write just $\Phi^{\vec{M}}$, when $\vec{M}$ is an oriented sublink of $L$. 

Following standard conventions, we identify a base orientation on $L$, and say that other orientations are obtained by reversing components relative to this base orientation.

Given a base orientation on $L$, we obtain a $\bH(L)$-valued Alexander grading on $\cCFL(L)$, and consequently on $\cC_{\vec{0}}$. If $\veps$ is an arbitrary index, then Equation~\eqref{eq:C_epsilon=CFL-L-veps} gives a $\bH(L_\veps)$-valued Alexander grading on $\cC_{\veps}$. We extend this to an Alexander grading taking values in $\bH(L_{\veps})\times \Z^{n-|L_{\veps}|}$, by declaring $A_j(T_i^{\pm 1})$  to be $\pm \delta_{ij}$ (the Kronecker delta). Note that the $\Z^{n-|L_\veps|}$ components are only well-defined up to an overall shift, though the $\bH(L_\veps)$ component of the grading is uniquely determined. Abusing notation slightly, we will write $A=(A_1,\dots, A_n)$ also for this extension of the Alexander grading. See \cite{ZemExact}*{Section~7.2} for a similar discussion.

There is  another natural normalization of the Alexander grading, for which we write $A^\sigma$ and refer to as the \emph{$\sigma$-normalized Alexander grading}. This grading takes values in $\bH(L)$. It is determined by the property that $A^\sigma$ coincides with $A^{\vec{0}}$ on $\cC_{\vec{0}}\iso \cCFL(L)$ and that the maps $\Phi^{\vec{M}}$ preserve $A^\sigma$ whenever $\vec{M}\subset L$ has only positively oriented components. 

If $\vec{M}$ is an oriented sublink of $L$, then
\[
A^\sigma_i(\Phi^{\vec{M}})=\frac{\lk(K_i,M)-\lk(K_i,\vec{M})}{2}.
\]
The quantity on the right can be reinterpreted as the sum over $-\lk(K_i,\vec{J})$, ranging over components $\vec{J}$ of $\vec{M}$ which are oriented negatively relative to the base orientation on $L$. (See \cite{MOIntegerSurgery}*{pg. 5}). 

\begin{lem}
There is a choice of absolute lift of the grading $A$, described above, which is given on $\cC_{\veps}$ by the formula
\[
A_i=A_i^\sigma+\begin{cases} \lk(K_i,L\setminus L_\veps)/2 & \text{ if } \veps_i=0\\
\lk(K_i,L\setminus K_i)/2& \text{ if } \veps_i=1.
\end{cases}
\] 
\end{lem}
\begin{proof} There are two non-trivial claims. The first is that if $\veps_i=0$, then $A_i$ agrees with the absolute Alexander grading from $\cCFL(L_\veps)$. The second is that if $\veps_i=1$, then $A_i$ takes values in $\Z$.

The second claim, i.e., that if $\veps_i=1$ then $A_i$ takes values in $\Z$ is straightforward, as $A_i^{\sigma}$ takes values in $\lk(K_i,L\setminus K_i)/2+\Z$ by definition. Therefore we focus on the first claim.

The first claim is clear if $\veps=\vec{0}$. We prove the claim by induction on $\veps$ and $i$. Assume that it is true for some $\veps$ and index $i$ such that $\veps_i=0$ and $\veps'=\veps+e_j$, for some $j\neq i$. The shift in the statement changes by
\[
\frac{\lk(K_i,L\setminus (L_{\veps}\cup K_j))}{2}-\frac{\lk(K_i,L\setminus L_{\veps})}{2}=-\frac{\lk(K_i,K_j)}{2}.
\]
On the other hand, this is exactly the shift in the absolute Alexander grading $A_i$ under the map $\Phi^{+K_j}$ by \cite{ZemBordered}*{Lemma~6.4 (4)}. (Implicitly, we are here using the fact that $\Phi^{K_j}$ can be described as the composition of the map which localizes at $Z_j$, followed by the composition of a grading preserving homotopy equivalence which sets $U_j=W_jZ_j$ and $Z_j=T_j$). This completes the proof.
\end{proof}
Using the above shifts, it is easy to verify that if $L_{\veps'}=L_{\veps}\sqcup M$, then
\begin{equation}
A_i(\Phi^{\vec{M}}_{\veps,\veps'})=\frac{1}{2}\cdot
\begin{cases}-\lk(K_i,\vec{M}) & \text{ if } \veps_i=\veps_i'=0\\
-\lk(K_i,L_{\veps}\setminus K_i)+\lk(K_i,M)-\lk(K_i,\vec{M}) & \text{ if } \veps_i=0, \veps_i'=1\\
\lk(K_i,M)-\lk(K_i,\vec{M})& \text{ if } \veps_i=\veps_i'=1.
\end{cases}
\label{eq:Alexander-grading-shifts}
\end{equation}

\begin{rem} The formula for the grading shift of $\Phi^{\vec{M}}$ with respect to $A^\sigma$ is simpler than with respect to $A$. However, the grading $A$ is simpler to write down on $\cC_{\veps}$, since it is induced by the absolute Alexander grading on $\cCFL(L_\veps)$. 
\end{rem}

%
%Suppose that $\veps\in \{0,1\}^n$ and $\veps'\ge \veps$. If $\veps_i=0$, then 
%\[
%A_i(\Phi^{\vec{L}_{\veps,\veps'}})=-\frac{\lk(K_i,L_{\veps,\veps'})}{2}.
%\]
%In the above, if $K_i$ is a component of $L_{\veps,\veps'}$, the linking number of $K_i$ with itself is computing using the framing from $\Lambda$.
%If $\veps_i=1$, then
%\begin{equation}
%A_i(\Phi^{\vec{L}_{\veps,\veps'}})=\sum_{-K_j\subset \vec{L}_{\veps,\veps'}} \lk(K_i,K_j).\label{eq:grading-shift-idmepotent-1}
%\end{equation}
%In the above, the sum on the right hand side is over the components of $\vec{L}_{\veps,\veps'}$ which are negatively oriented relative to the base orientation on $L$.
%
%\begin{rem} Manolescu and Ozsv\'{a}th use a different grading convention. They view all of the groups $\cC_{\veps}$ as being graded by $\bH(L)$. With respect to this convention, all of the maps $\Phi^{\vec{L}_{\veps,\veps'}}$ have Alexander grading shift given in Equation~\eqref{eq:grading-shift-idmepotent-1}. We choose the grading convention above, since it makes our description of the modules below somewhat more natural. 
%\end{rem}

\subsection{The surgery algebra $\cK$}

We now recall the surgery algebra from \cite{ZemBordered,ZemExact}. We write $\ve{I}=\ve{I}_0\oplus \ve{I}_1$ for an idempotent ring on two idempotents, where each $\ve{I}_{\veps}\iso\bF = \Z/2$. The surgery algebra $\cK$ is defined as follows. We set
\[
\ve{I}_0\cdot \cK\cdot \ve{I}_0=\bF[W,Z], \quad \ve{I}_1\cdot \cK\cdot \ve{I}_1=\bF[U,T,T^{-1}],\quad \ve{I}_0\cdot  \cK \cdot \ve{I}_1=0.
\]
 We declare
 \[
\ve{I}_1\cdot \cK\cdot \ve{I}_0=\bF[U,T,T^{-1}]\otimes_{\bF} \langle \sigma, \tau\rangle.
 \]
 The generators $\sigma$ and $\tau$ satisfy
 \[
 \sigma \cdot W^i Z^j=U^i T^{j-i}\cdot \sigma\quad \text{and}\quad \tau \cdot W^i Z^j=U^j T^{j-i}\cdot\tau.
 \]
 
 It is sometimes helpful to also describe two algebra morphisms
 \begin{equation}
 \phi^\sigma, \phi^\tau \colon \bF[W,Z]\to \bF[U,T,T^{-1}]\label{eq:phi-sig-tau}
 \end{equation}
 given by the formulas
 \[
 \phi^\sigma(W^iZ^j)=U^i T^{j-i}\quad \text{and} \quad \phi^\tau(W^iZ^j)=U^j T^{j-i}.
 \]
 We can restate the definitions of $\phi^\sigma$ and $\phi^\tau$ in terms of these maps as
 \[
 \sigma\cdot a=\phi^\sigma(a)\cdot \sigma\quad \text{and} \quad \tau \cdot a=\phi^\tau(a)\cdot\tau. 
 \]
 
 It is helpful to write
 \[
 \cR_0:=\ve{I}_0\cdot \cK\cdot \ve{I}_0\quad \text{and} \quad \cR_1:=\ve{I}_1\cdot \cK\cdot \ve{I}_1. 
 \]

\subsection{The Koszul dual algebra $\cK^!$}

We now recall the Koszul dual algebra $\cK^!$, described by the second author in \cite{ZemKoszul}. Similar to $\cK$, this $\cK^!$ is an algebra over the idempotent ring $\ve{I}=\ve{I}_0\oplus \ve{I}_1$. In contrast to $\cK$,  $\cK^!$ is a curved $dg$-algebra.

We declare
\[
\ve{I}_0\cdot \cK^!\cdot \ve{I}_0=\frac{\bF\langle w,z,\theta\rangle}{w^2=z^2=[w,\theta]=[z,\theta]=\theta^2=0}.
\]
In the above, $\bF\langle w,z,\theta\rangle$ denotes the free, noncommutative algebra on the symbols $w$, $z$ and $\theta$.

We declare $\ve{I}_1\cdot \cK^!\cdot \ve{I}_0=0$.

Next, we set
\[
\ve{I}_1\cdot \cK^! \cdot \ve{I}_1=\frac{\bF\langle \varphi_+,\varphi_-,\theta\rangle}{\varphi_+^2=\varphi_-^2=[\theta,\varphi_\pm]=\theta^2=0}.
\]

We define $\ve{I}_0\cdot \cK^!\cdot \ve{I}_1$ to be generated as a bimodule over $(\ve{I}_0\cdot \cK^!\cdot \ve{I}_0, \ve{I}_1\cdot \cK^!\cdot \ve{I}_1)$ by two algebra elements $s$ and $t$. We define the product of two monomials to either be zero or be given by concatenation, subject to the following relations:
\[
ws=s\varphi_-=0,\quad zt=t\varphi_+=0,\quad s\theta=\theta s,\quad zs=s\varphi_+,\quad t\theta=\theta t,\quad \text{and} \quad wt=t\varphi_-.
\]
As a vector space, we have
\[
\ve{I}_0\cdot \cK^!\cdot \ve{I}_1=
\Span_{\bF}(s,s\theta,zs,zs\theta,t,t\theta,wt,wt\theta).
\]

Additionally, the algebra $\cK^!$ has a differential $\mu_1$, which is supported only in idempotent $(0,0)$, and which is given by
\[
\mu_1(\theta)=wz+zw,
\]
extended to other monomials via the Leibniz rule. Finally, $\cK^!$ has a curvature term, supported in idempotent $(1,1)$, given by
\[
\mu_0=\varphi_-\varphi_++\varphi_+\varphi_-.
\]

It is convenient to write
\[
\cR_0^!:=\ve{I}_0\cdot \cK^!\cdot \ve{I}_0\quad \text{and} \quad \cR_1^!:=\ve{I}_1\cdot \cK^!\cdot \ve{I}_1. 
\]

In \cite{ZemKoszul}, the second author constructs rank 1 dualizing bimodules ${}^{\cK^!} [\cCo]^{\cK}$ and ${}_{\cK}[\cTr]_{\cK^!}$, which give equivalences of categories between suitable categories of modules over $\cK$ and over $\cK^!$. The bimodule ${}^{\cK^!} [\cCo]^{\cK}$ is given by the following diagram:
\[
\begin{tikzcd}[labels=description, row sep=2cm]
i_0  \ar[loop left,looseness=20, "w|W"] \ar[loop above,looseness=20, "\theta|U"] \ar[loop right,looseness=20, "z|Z"]
\ar[d, bend left, "s|\sigma"]
\ar[d, bend right, "t|\tau"]
\\
i_1 \ar[loop left,looseness=20, "\varphi_-|T^{-1}"] \ar[loop below,looseness=20, "\theta|U"] \ar[loop right,looseness=20, "\varphi_+|T"]
\end{tikzcd}
\]
The bimodule ${}_{\cK} [\cTr]_{\cK^!}$ is more complicated, and more details can be found in \cite{ZemKoszul}.

We will define the \emph{tensor weight} (or just \emph{weight}) of a monomial in $\cK^!$ to be the number of letters. For example
\[
\wt(wzw)=3\quad \wt(s)=1\quad \wt(wzwz\theta)=5.
\]
It is convenient to complete $\cK^!$ over the weight filtration. For our purposes, it is actually more convenient to complete just $\cR_0^!$ under this weight filtration. We write $\ve{\cK}^!$ for the algebra with $\cR_0^!$ completed in this manner. We write 
\[
\bm{\cR}_0^!:=\ve{I}_0\cdot \ve{\cK}^!\cdot \ve{I}_0.
\]
\begin{rem}
Given a type-$A$ module ${}_{\cK} \cX$, one can tensor with the cotrace bimodule ${}^{\cK^!} [\cCo]^{\cK}$ to obtain a type-$D$ module ${}^{\cK^!} \cX$. The reader should think of the structure map $\delta^1$ on ${}^{\cK^!} \cX$ as encoding the following information. The $w$ weighted component of $\delta^1$ encodes the map $m_2(W,-)$. The $\theta$-weighted component encodes $m_2(U,-)$. More generally, if $a\in \cK^!$ is a monomial in $\cK^!$ with tensor weight $n$, then the $a$ component of $\delta^1$ is equal to the sum of 
$m_{n+1}(a_1^!,\dots, a_n^!,-)$, ranging over all sequences $a_n,\dots, a_1\in \{w,z,\theta,t,\tau,\varphi_-,\varphi_+\}$ such that $a_n\cdots a_1=a$. Here $a_i^!\in \cR_0$ denotes the dual element to $a_i\in \cR_0^!$. (E.g. if $a_i=w$ then $a_i^!=W$, and so forth). We think of the type-$D$ structure ${}^{\cK^!} \cX$ as encoding just a small collection of the actions of the $A_\infty$-module ${}_{\cK} \cX$, from which the entire module ${}_{\cK} \cX$ can be recovered up to homotopy equivalence.
\end{rem}

As discussed in the introduction, the following definition is important for our purposes:

\begin{define}
\,
\begin{enumerate}
\item  If ${}_{\cK} \cX^{\cK}$ is a $DA$-bimodule, we say that $\cX$ is \emph{$U$-equivariant} if the structure maps satisfy
\[
\delta_{n+1}^1(a_n,\dots, Ua_i,\dots, a_1,\xs)=\delta_{n+1}^1(a_n,\dots, a_1,\xs)\cdot U,
\]
where $\cdot U$ means to multiply the $\cK$ factor of $\delta_{n+1}(a_n,\dots, a_1,\xs)$ by $U$. A morphism of $DA$-bimodules is called $U$-equivariant if the analogous condition is satisfied. 
\item If ${}^{\cK^!} \cX^{\cK}$ is a $DD$-bimodule, we say that $\cX$ is \emph{$U$-equivariant} if the structure map $\delta^{1,1}$ has a component of the form $\theta\otimes \id\otimes U$ (i.e. a self arrow on each generator weighted by $\theta|U$), and has no other components weighted by a multiple of $\theta$. We say a morphism of $DD$-bimodules over $(\cK^!,\cK)$ is \emph{$U$-equivariant} if there are no components weighted by a multiple of $\theta$. 
\end{enumerate} 
\label{def:U equivariant modules}
\end{define}

It is straightforward to see that if ${}_{\cK} \cX^{\cK}$ has $U$-equivariant structure maps and is strictly unital, then ${}^{\cK^!} [\cTr]^{\cK}\boxtimes {}_{\cK} \cX^{\cK}$ also has $U$-equivariant structure maps. Similarly, if $f_*^1$ is a strictly unital and $U$-equivariant morphism of $DA$-bimodules, then $\bI_{[\cCo]}\boxtimes f_*^1$ has no $\theta$-weighted terms, and hence is $U$-equivariant in the above sense.

 Slightly less obviously, it is also the case that if ${}^{\cK^!} \cY^{\cK}$ has $U$-equivariant structure maps (in the above sense), then ${}_{\cK} [\cTr]_{\cK^!} \boxtimes {}^{\cK^!} \cY^{\cK}$ is strictly unital and has $U$-equivariant structure maps. See \cite{ZemKoszul}*{Corollary~5.8}.  

\section{Staircase complexes}

An important component of the proofs from \cite{CZZSatellites} was an analysis of staircase complexes, as well as morphisms between staircase complexes.

\begin{define} A \emph{staircase complex} $\cC$ is a free chain complex over $\bF[W,Z]$ with generators $\xs_0,\xs_1,\dots, \xs_{2n}$ such that $\d(\xs_{2i})=0$ and
\[
\d(\xs_{2i+1})=\xs_{2i}W^{\a_i}+\xs_{2i+2}Z^{\b_i},
\]
for some $\a_i,\b_i>0$.
\end{define}

%We say that a staircase complex has \emph{standard gradings} if $\gr_{\ws}(\xs_0)=\gr_{\zs}(\xs_{2n})=0$. If $\cC$ is a staircase complex with standard gradings, then the homology $H_*(\cC)$ is naturally isomorphic to a monomial ideal $I_{\cC}\subset \bF[W,Z]$. The monomial $I_{\cC}$ is spanned by $W^{i}Z^j$ for each generator $\xs_{2k}$ of $\cC$ with $\gr_{\ws}(\xs_{2k})=-2i$ and $\gr_{\zs}(\xs_{2k})=-2j$.

If $\cC=(\d\colon \cC_1\to \cC_0)$ is a staircase complex, we define the \emph{algebraic grading} on $\cC$ to be $1$ on $\cC_1$ and $0$ on $\cC_0$. 

Another perspective on staircase complexes is that they are free-resolutions of their homology. That is, if $\cC=(\d\colon \cC_1\to \cC_0)$ is a staircase complex, then the following sequence is exact:
\[
\begin{tikzcd}
0\ar[r]&\cC_0\ar[r, "\d"] &\cC_1 \ar[r, twoheadrightarrow] & H_*(\cC)\ar[r] &0.
\end{tikzcd}
\]

\begin{lem} 
\label{lem:basic-staircase-lemma}
Suppose that $\cA$ and $\cB$ are two staircase complexes and $F\colon H_*(\cA)\to H_*(\cB)$ is a homogeneously graded $\bF[W,Z]$-equivariant map. Then there is a chain map $f\colon \cA\to \cB$ which induces $F$ on homology and preserves the algebraic grading. Furthermore, the map $f$ is unique up to chain homotopy. In particular, if $f,g\colon \cA\to \cB$ are two homogeneously graded maps which preserve the algebraic grading and which induce the same map on homology, then $f$ and $g$ are chain homotopic. 
\end{lem}

Since $\cA$ and $\cB$ are free resolutions of their homology, the above result follows from basic homological algebra. (See \cite{Weibel}*{Theorem 2.2.6}).

We will have use for some slightly less obvious results about staircase complexes:

\begin{lem}[\cite{CZZSatellites}*{Lemma~5.19}]
\label{lem:ext-computation} If $\cS$ is a staircase complex and $f\colon \cS\to \cS$ is a chain map which shifts the $(\gr_{\ws},\gr_{\zs})$-grading by $(-1,-1)$, then $f$ is null-homotopic.
\end{lem}

\begin{rem} If we write $\cS=\Cone(\d\colon \cS_1\to \cS_0)$, then an endomorphism $\phi$ of $\cS$ with $(\gr_{\ws},\gr_{\zs})(\phi)=(-1,-1)$ can be written as a sum $\phi=\phi_{10}+\phi_{01}$ where $\phi_{ij}$ maps $\cS_j$ to $\cS_i$. If $\phi$ is a chain map, then it is easy to see that $\phi_{10}=0$, because $\d$ is injective as a map from $\cS_1$ to $\cS_0$. The component $\phi_{01}$ naturally determines an element in the group $\Ext^1_{\cR_0}(\cS,\cS)$. The above lemma can be interpreted as the vanishing of the subspace of this $\Ext$ group in $(\gr_{\ws},\gr_{\zs})$-grading $(-1,-1)$. 
\end{rem}
 
 Lemma~\ref{lem:ext-computation} follows from a slightly more general result, which we will periodically find useful:
 
 \begin{lem}\label{lem:generalized-Ext-computation} Suppose that $\cA$ and $\cB$ are staircase complexes with absolute $(\gr_{\ws},\gr_{\zs})$-gradings and such that there is a grading preserving chain map $I\colon \cA\to \cB$ which is non-zero on homology. If $\phi\colon \cA\to \cB$ is a chain map of grading $(-1,-1)$, then $\phi$ is null-homotopic.
 \end{lem}
 \begin{proof} The proof is essentially the same as \cite{CZZSatellites}*{Lemma~5.19}.  Write $\cA$ as $\Cone(\d_{\cA}\colon \cA_1\to \cA_0)$ and $\cB$ as $\Cone(\d_{\cB}\colon \cB_1\to \cB_0)$. The map $\phi$ can be written as $\phi=\phi_{01}+\phi_{10}$, where $\phi_{ij}$ maps $\cA_j$ to $\cB_i$. Since $\d_{\cA}$ and $\d_{\cB}$ are injective, it is easy to see that $\phi_{10}$ is zero, so it suffices to show that $\phi=\phi_{01}$ is null-homotopic.
 
 Let $\xs_0,\xs_2,\dots, \xs_{2n}$ be the generators for $\cA_0$, and $\xs_1,\xs_3,\dots, \xs_{2n-1}$ be the generators for $\cA_1$. Let $i$ be the minimal index such that $\phi$ is non-zero on $\xs_i$. The proof is by induction on $i$. 
 
We consider two cases: $\phi(\xs_i)$ is zero in homology, and $\phi(\xs_i)$ is non-zero in homology.  In the first case, if $\phi(\xs_i)=\d(\ys)$ for some $\ys\in \cB_1$, then we define a homotopy $j$ which sends $\xs_i$ to $\ys$. The map $\phi+\d(j)$ vanishes on $\xs_0,\dots, \xs_i,\xs_{i+1}$ (i.e. we have increased the minimal index on which $\phi$ is non-vanishing by 2). We now consider the case that $\phi(\xs_i)$ is non-zero in homology. Write $\d(\xs_i)=\xs_{i-1} W^n+\xs_{i+1} Z^m$. Now $\gr_{\ws}(\xs_{i+1}Z^{m})$ has the same grading as $\phi(\xs_{i})$. We define a map $j\colon \cA\to \cB$ which sends $\xs_{i+1}$ to $I(\xs_{i+1})Z^m$. Now $\phi+\d(j)$ vanishes on $\xs_0,\dots, \xs_{i-1}$, and now sends $\xs_{i}$ to a boundary. Using the previous argument, we may modify $\phi$ by a chain homotopy so that it vanishes on $\xs_0,\dots, \xs_{i+1}$. By induction, the proof is complete.
 \end{proof}

Another helpful lemma is the following:

\begin{lem}
\label{lem:grading-preserving-endo=id} Suppose that $\cS$ is a staircase complex and $f\colon \cS\to \cS$ is a grading preserving chain homotopy equivalence. Then $f=\id$.
\end{lem}
\begin{proof}Write $\cS=(\d\colon \cS_1\to \cS_0)$. Since $f$ preserves the Maslov bigrading, it must also preserve the algebraic grading. Write $f_i=f|_{\cS_i}$. Since $H_*(\cS)$ has rank 1 or 0 in each $(\gr_{\ws},\gr_{\zs})$-bigrading, the map $f$ must induce the identity map on homology. We observe that for each generator $\xs_{2i}$ in algebraic grading $0$, there are no other non-zero elements of $\cS$ in the same $(\gr_{\ws},\gr_{\zs})$-grading. Therefore $f(\xs_{2i})=\xs_{2i}$. Since $f$ is a chain map, we know that
\[
\d\circ f_1=f_0\circ \d.
\]
Since $f_0=\id$, we conclude that
\[
\d\circ (f_1+\id)=0.
\]
Since $\d$ is injective as a map from $\cS_1$ to $\cS_0$, we conclude $f_1=\id$. 
\end{proof}

An additional result that we will use concerns endomorphisms of staircases of positive degree:

\begin{lem}\label{lem:staircase-no-maps-positive bidegree} Let $\cS$ be a staircase complex and $\phi\colon \cS\to \cS$ be a map which has $(\gr_{\ws},\gr_{\zs})$-grading $(n,n)$. If $n\ge 1$, then $\phi=0$. 
\end{lem}
\begin{proof}Write $\xs_0,\xs_1,\dots, \xs_{2n}$ for the generators of $\cS$, ordered so that $\d(\xs_{2i+1})= \xs_{2i}W^{\a_i}+\xs_{2i+2}Z^{\b_i}$. We observe that
\[
\gr_{\ws}(\xs_{i+1})<\gr_{\ws}(\xs_{i})\quad  \text{and} \quad\gr_{\zs}(\xs_{i+1})>\gr_{\zs}(\xs_{i})
\]
for all $i$. Suppose $\phi(\xs_i)$ contains a summand of $\xs_jW^{a}Z^b$. Suppose that $j\ge i$. Since $\phi$ has $\gr_{\ws}$-grading $n$,  we have $\gr_{\ws}(\xs_j)-2a=\gr_{\ws}(\xs_i)+n$, or equivalently
\[
\gr_{\ws}(\xs_j)-\gr_{\ws}(\xs_i)=n+2a.
\]
Since $j\ge i$, the left hand side is non-positive, while the right hand side is positive, so this is impossible. If $j\le i$, the same argument applies with $\gr_{\zs}$ replacing $\gr_{\ws}$. This completes the proof.
\end{proof}

\section{The candidate bimodule}

In this section, we describe our candidate bimodule for a 2-component L-space link, equipped with framing $(0,0)$. We will write ${}^{\cK^!} \cX^{\alg}(L)^{\cK}$ for this bimodule.  The candidate $DA$-bimodule, ${}_{\cK} \cX^{\alg}(L)^{\cK}$, is obtained by tensoring with the trace bimodule from Koszul duality (constructed in \cite{ZemKoszul}):
\[
{}_{\cK} \cX^{\alg}(L)^{\cK}:={}_{\cK} [\cTr]_{\cK^!} \boxtimes {}^{\cK^!} \cX^{\alg}(L)^{\cK}.
\]

\subsection{The complexes at each idempotent}

\label{sec:complexes-at-idempotents}
\subsubsection{Idempotent $(0,0)$}

We now describe the candidate module ${}^{\cK^!} \cX^{\alg}(L)^{\cK}$ in idempotent $(0,0)$. We will view this as a type-$DD$-module over $(\cR_0^!,\cR_0)$, and will write ${}^{\cR_0^!} \cX^{\alg}(L)^{\cR_0}$. 

 Schematically, the bimodule ${}^{\cR_0^!} \cX^{\alg}(L)^{\cR_0}$ takes the following form:
\begin{equation}
\begin{tikzcd}[column sep=1.5cm]
\cdots
\cC_{s_1-2} 
	\ar[r, "z|L_z", bend left=30]
	\ar[loop above, "\substack{zw|L_{zw}+\\wz|L_{wz}+\\\theta|U}"]
& \cC_{s_1-1}
	\ar[r, "z|L_z", bend left=40]
	\ar[l, "w|L_w", bend left=30]
	\ar[loop above, "\substack{zw|L_{zw}+\\wz|L_{wz}+\\\theta|U}"]
& \cC_{s_1}
	\ar[r, "z|L_z", bend left=40]
	\ar[l, "w|L_w", bend left=40]
	\ar[loop above, "\substack{zw|L_{zw}+\\wz|L_{wz}+\\\theta|U}"]
& \cC_{s_1+1}
	\ar[r, "z|L_z", bend left=30]
	\ar[l, "w|L_w", bend left=40]
	\ar[loop above, "\substack{zw|L_{zw}+\\wz|L_{wz}+\\\theta|U}"]
& \cC_{s_1+2}\cdots
	\ar[l, "w|L_w", bend left=30]
	\ar[loop above, "\substack{zw|L_{zw}+\\wz|L_{wz}+\\\theta|U}"]
\end{tikzcd}.
\label{sec:candidate-idempotent-00}
\end{equation}
In the above, each of the complexes $\cC_{s_1}$ is a type-$D$ module over $\cR_0$. We define each $\cC_{s_1}$ to be the canonical 2 step resolution of the module $\bigoplus_{s_2\in \Z+\lk(K_1,K_2)/2}\cHFL(L, s_1,s_2)$. This is a torsion free complex by definition of an L-space link. As in Section~\ref{sec:L-space-links-background}, the free-resolution of this complex is a staircase complex. 

In Equation~\eqref{sec:candidate-idempotent-00}, we choose the maps $L_w\colon \cC_{s_1}\to \cC_{s_1-1}$ and $L_z\colon \cC_{s_1}\to \cC_{s_1+1}$ to be any chain maps which have $(\gr_{\ws}, \gr_{\zs})$-grading shift  equal to $(-2,0)$ and $(0,-2)$, respectively, and which are non-zero on homology. Any two choices of maps for $L_w$ are chain homotopic as type-$D$ morphisms by Lemma~\ref{lem:basic-staircase-lemma}, and similarly for $L_z$.

Before defining $L_{zw}$ and $L_{wz}$, we set some notation. If $f\colon \cC_{s_1}\to \cC_{s_1'}$ is a type-$D$ morphism, we write $\d_{\Int}(f)$ for the morphism differential of $f$ as a map of type-$D$ modules. 

We declare $L_{zw}$ and $L_{wz}$ to be any choice of maps such that
\[
\d_{\Int}(L_{zw})=L_w\circ L_z+\id\otimes U\quad \text{and} \quad \d_{\Int}(L_{wz})=L_z\circ L_w+\id\otimes U,
\]
which have $(\gr_{\ws},\gr_{\zs})$-grading equal to $(-1,-1)$, and which increase the algebraic grading (i.e. staircase grading) of $\cC_{s_1}$ by 1.

\begin{rem}The above construction is essentially the same as the one appearing in \cite{CZZSatellites}*{Section~5.5}, with the maps $L_w$, $L_z$, $L_{wz}$ and $L_{zw}$ of our present work being the same as $L_W$, $L_Z$, $h_{Z,W}$ and $h_{W,Z}$ in \cite{CZZSatellites}. The relation between the perspective in \cite{CZZSatellites} and the above $DD$-bimodule is explained in \cite{ZemKoszul}*{Section~5.3}. 
\end{rem}

A similar argument to \cite{CZZSatellites}*{Proposition~5.9} shows that the induced bimodule ${}^{\cR_0^!} \cX^{\alg}(L)^{\cR_0}$ is well-defined up to homotopy equivalence.

The groups $\bigoplus_{s_2\in \Z+\lk(K_1,K_2)/2}\cHFL(L, s_1,s_2)$ can be read off from the $H_L$-function, as we now recall. Each of the groups $\cHFL(L,s_1,s_2)$ is isomorphic to $\bF[U]$ since $L$ is an L-space link. By definition, the generator of $\cHFL(L,s_1,s_2)$ has $\gr_{\ws}$ grading $-2H_L(s_1,s_2)$ and $\gr_{\zs}$ grading $-2H_L(s_1,s_2)-2s_1-2s_2$. See, e.g., \cite{CZZApplications}*{Section~4.1}. 

\subsubsection{Idempotent $(1,0)$}

We now describe the complexes appearing at idempotent $(1,0)$. Let $\cS$ be the staircase complex $\cCFK(K_2)$, viewed as a type-$D$ module over $\bF[W,Z]$. We define the complex ${}^{\cR_1^!}\cX^{\alg}(L)^{\cR_0}$ to be the following:
\[
\begin{tikzcd}[column sep=1.5cm]
\cdots
T^{s_1-2} \cS
	\ar[r, "\varphi_+|1", bend left=30]
	\ar[loop below,"\theta|U"]
& T^{s_1-1}\cS
	\ar[r, "\varphi_+|1", bend left=30]
	\ar[l, "\varphi_-|1", bend left=30]
	\ar[loop below,"\theta|U"]
& T^{s_1} \cS
	\ar[r, "\varphi_+|1", bend left=30]
	\ar[l, "\varphi_-|1", bend left=30]
	\ar[loop below,"\theta|U"]
& T^{s_1+1} \cS
	\ar[r, "\varphi_+|1", bend left=30]
	\ar[l, "\varphi_-|1", bend left=30]
	\ar[loop below,"\theta|U"]
& T^{s_1+2} \cS\cdots
	\ar[l, "\varphi_+|1", bend left=30]
	\ar[loop below,"\theta|U"]
\end{tikzcd}
\]
In the above, $s_1$ takes values in $\Z$.

We equip each $T^{s_1}\cS$ with Maslov gradings $\gr_{\ws},\gr_{\zs}$ by identifying it with $\cCFK(K_2)$. We define the $A_2$ Alexander grading $A_2= \frac{1}{2}(\gr_{\ws}-\gr_{\zs})$. We view $T^{s_1} \cS$ as being in $A_1$ Alexander grading $s_1$.

\subsubsection{Idempotent $(0,1)$}

We now consider idempotent $(0,1)$. We define the complex ${}^{\cR^!_0}\cX^{\alg}(L)^{\cR_1}$ abstractly as ${}^{\cR^!_0}\cX(K_1)^{\bF[U]}\otimes [1]^{\bF[T,T^{-1}]}$, where ${}^{\cR_0^!} \cX(K_1)^{\bF[U]}$ is the idempotent 0 subspace of the surgery complex of the knot $K_1$ (treated as a $DD$-bimodule over $(\cR_0^!,\bF[U])$.  In more detail, the $AA$-bimodule ${}_{\cK} \cX(K_1)_{\bF[U]}$ is described in \cite{ZemBordered}*{Section~8}. The construction therein is naturally free on the right action of $U$ (with the left $\cK$-action being $U$-equivariant), and therefore this naturally describes a type-$DA$-bimodule ${}_{\cK} \cX(K_1)^{\bF[U]}$. We define the $DD$-bimodule ${}^{\cR_0^!} \cX(K_1)^{\bF[U]}$ by restricting the left idempotent, and tensoring with ${}^{\cR_0^!} [\cCo]^{\cR_0}$.  The module $[1]^{\bF[T,T^{-1}]}$ denotes a type-$D$ structure with one generator, and vanishing differential.

Using the fact that $K_1$ is an L-space knot, we may produce the following model, which is smaller and the the one we focus on. Write $H_{K_1}\colon \Z\to \Z^{\ge0 }$ for the $H$-function of $K_1$. We will set
\[
\eta(s_1):=H_{K_1}(s_1)-H_{K_1}(s_1+1),
\]
which always takes value in $\{0,1\}$. Then ${}^{\cR^!_0}\cX^{\alg}(L)^{\cR_1}$ is given by the diagram
\[
\begin{tikzcd}[column sep=2cm]
\cdots
\ys_{-2}
	\ar[r, "z|U^{\eta(-2)}", bend left=30]
	\ar[loop below,"\theta|U"]
& \ys_{-1}
	\ar[r, "z|U^{\eta(-1)}", bend left=30]
	\ar[l, "w|U^{1-\eta(-2)}", bend left=30]
	\ar[loop below,"\theta|U"]
& \ys_0
	\ar[r, "z|U^{\eta(0)}", bend left=30]
	\ar[l, "w|U^{1-\eta(-1)}", bend left=30]
	\ar[loop below,"\theta|U"]
& \ys_1
	\ar[r, "z|U^{\eta(1)}", bend left=30]
	\ar[l, "w|U^{1-\eta(0)}", bend left=30]
	\ar[loop below,"\theta|U"]
& \ys_2\cdots
	\ar[l, "w|U^{1-\eta(1)}", bend left=30]
	\ar[loop below,"\theta|U"]
\end{tikzcd}
\]
It is convenient to view ${}^{\cR_0^!} \cX^{\alg}(L)^{\cR_1}$ as having the bigrading $(\gr_{\ws},\gr_{\zs})$, where 
\[
(\gr_{\ws},\gr_{\zs})(\ys_{s_1})=(-2H_{K_1}(s_1),-2H_{K_1}(s_1)-2s_1).
\]
We set
\[
(A_1,A_2)(\ys_{s_1})=(s_1,0).
\]

\subsubsection{Idempotent $(1,1)$}

Finally, we describe the complexes ${}^{\cR_1^!} \cX^{\alg}(L)^{\cR_1}$. We define this complex as the following diagram:
\[
\begin{tikzcd}[column sep=1.5cm]
\cdots
T^{-2}\ve{e}
	\ar[r, "\varphi_+|1", bend left=30]
	\ar[loop below,"\theta|U"]
& T^{-1}\ve{e}
	\ar[r, "\varphi_+|1", bend left=30]
	\ar[l, "\varphi_-|1", bend left=30]
	\ar[loop below,"\theta|U"]
& T^{0}\ve{e}
	\ar[r, "\varphi_+|1", bend left=30]
	\ar[l, "\varphi_-|1", bend left=30]
	\ar[loop below,"\theta|U"]
& T^{1}\ve{e}
	\ar[r, "\varphi_+|1", bend left=30]
	\ar[l, "\varphi_-|1", bend left=30]
	\ar[loop below,"\theta|U"]
& T^{2}\ve{e}
	\ar[l, "\varphi_+|1", bend left=30]
	\ar[loop below,"\theta|U"]
\end{tikzcd}
\]
We view the generator $T^{s_1}\ve{e}$ as having a single Maslov grading, $\gr$, where $\gr(T^{s_1}\ve{e})=0$. We view
\[
(A_1,A_2)(T^{s_1}\ve{e})=(s_1,0).
\]

\subsection{The length 1 maps}

We now define the length 1 maps of the surgery bimodule. In our bimodule ${}^{\cK^!} \cX^{\alg}(L)^{\cK}$, these are the components which shift only one idempotent. (The term \emph{length} comes from the perspective of the link surgery complex from \cite{MOIntegerSurgery} as a hypercube, where the ``length'' refers to the $L^1$-length of the corresponding line segment in the cube).

\subsubsection{Idempotent $(0,0)$ to $(1,0)$}
\label{sec:vertical map from (0,0) to (1,0)}
The maps from idempotent $(0,0)$ to idempotent $(1,0)$ are described (in slightly different language) in \cite{CZZSatellites}*{Section~5.4}. We repeat the description in our present notation. We pick maps 
\[
L_s, L_t\colon \cC_{s_1}\to T^{s_1-\ell/2}\otimes \cS
\]
which are chain maps and are non-zero on homology,  preserve the staircase grading, and which have the following Maslov gradings:
\[
(\gr_{\ws},\gr_{\zs})(L_s)=(0,-2s_1+\ell)\quad \text{and} \quad (\gr_{\ws},\gr_{\zs})(L_t)=(2s_1+\ell,0).
\]

Additionally, we pick maps
\[
L_{zs}\colon \cC_{s_1}\to T^{s_1-\ell/2+1}\cS\quad \text{and} \quad  L_{wt}\colon \cC_{s_1}\to T^{s_1-\ell/2-1} \cS.
\]
We assume that these maps increase the algebraic (staircase) grading by 1, and have grading
\[
(\gr_{\ws},\gr_{\zs})(L_{zs})=(1,-2s_1+\ell-1)\quad \text{and} \quad (\gr_{\ws},\gr_{\zs})(L_{wt})=(2s_1+\ell-1,1).
\]
Additionally, we require these maps to satisfy
\[
\d_{\Int}(L_{zs})=L_s\circ L_z +L_{\varphi_+}\circ L_\sigma \quad \text{and} \quad \d_{\Int}(h_{wt})=L_t \circ L_w+ L_{\varphi_-}\circ L_t.
\] 
In the above, $L_{\varphi_{\pm}}\colon T^{s_1}  \cS\to T^{s_1\pm 1} \otimes \cS$ are the canonical identifications.

We form the components of $\delta^{1,1}$ from idempotent $(0,0)$ to $(1,0)$ by using the above maps. The maps $L_s$ are weighted by $s$. The maps $L_t$ are weighted by $t$. The maps $L_{zs}$ are weighted by $zs$. The maps $L_{wt}$ are weighted by $wt$. We arrange these as shown below:
\[
\begin{tikzcd}[column sep=1cm, row sep=3cm,labels=description]
\cdots
\cC_{s_1-2} 
	\ar[r, "z|L_z", bend left=30]
	\ar[loop above, "\substack{zw|L_{zw}+\\wz|L_{wz}+\\\theta|U}", looseness=25]
	\ar[d, "\substack{s|L_s\\ + t|L_t}"]
	\ar[dr,"zs|L_{zs}", pos=.7]
& \cC_{s_1-1}
	\ar[r, "z|L_z", bend left=40]
	\ar[l, "w|L_w", bend left=30]
	\ar[loop above, "\substack{zw|L_{zw}+\\wz|L_{wz}+\\\theta|U}", looseness=25]
	\ar[d, "\substack{s|L_s\\ + t|L_t}"]
		\ar[dr,"zs|L_{zs}",pos=.7]
	\ar[dl,crossing over, "wt|L_{wt}", pos=.3 ]
& \cC_{s_1}
	\ar[r, "z|L_z", bend left=40]
	\ar[l, "w|L_w", bend left=40]
	\ar[loop above, "\substack{zw|L_{zw}+\\wz|L_{wz}+\\\theta|U}", looseness=25]
	\ar[d, "\substack{s|L_s\\ + t|L_t}"]
	\ar[dr,"zs|L_{zs}",pos=.7]
	\ar[dl,crossing over, "wt|L_{wt}", pos=.3 ]
& \cC_{s_1+1}
	\ar[r, "z|L_z", bend left=30]
	\ar[l, "w|L_w", bend left=40]
	\ar[loop above, "\substack{zw|L_{zw}+\\wz|L_{wz}+\\\theta|U}", looseness=25]
	\ar[d, "\substack{s|L_s\\ + t|L_t}"]
	\ar[dr,"zs|L_{zs}",pos=.7]
	\ar[dl,crossing over, "wt|L_{wt}", pos=.3 ]
& \cC_{s_1+2}\cdots
	\ar[l, "w|L_w", bend left=30]
	\ar[loop above, "\substack{zw|L_{zw}+\\wz|L_{wz}+\\\theta|U}", looseness=25]
	\ar[d, "\substack{s|L_s\\ + t|L_t}"]
	\ar[dl,crossing over, "wt|L_{wt}" ,pos=.3]
\\
\cdots
T^{s_1-\ell/2-2} \cS
	\ar[r, "\varphi_+1", bend left=30]
	\ar[loop below,"\theta|U", looseness=15]
& T^{s_1-\ell/2-1}\cS
	\ar[r, "\varphi_+|1", bend left=30]
	\ar[l, "\varphi_-|1", bend left=30]
	\ar[loop below,"\theta|U", looseness=15]
& T^{s_1-\ell/2} \cS
	\ar[r, "\varphi_+|1", bend left=30]
	\ar[l, "\varphi_-|1", bend left=30]
	\ar[loop below,"\theta|U", looseness=15]
& T^{s_1-\ell/2+1} \cS
	\ar[r, "\varphi_+|1", bend left=30]
	\ar[l, "\varphi_-|1", bend left=30]
	\ar[loop below,"\theta|U", looseness=15]
& T^{s_1-\ell/2+2} \cS\cdots
	\ar[l, "\varphi_+|1", bend left=30]
	\ar[loop below,"\theta|U", looseness=15]
\end{tikzcd}
\]

It is convenient to write $\Psi^{K_1}$ for the components of the above diagram which are weighted by multiples of $s$, and to write $\Psi^{-K_1}$ for the components of the above diagram which are weighted by multiples of $t$. 

\begin{rem} In \cite{CZZSatellites}*{Section~5.5}, we wrote $L_{\sigma}$, $L_{\tau}$, $h_{\sigma,Z}$ and $h_{\tau,W}$ for the maps we write $L_{s}$, $L_t$, $L_{zs}$, and $L_{wt}$ above (respectively). In \cite{CZZSatellites}*{Section~5.5}, we also included choices of homotopies $h_{\sigma,W}$ and $h_{\tau,Z}$, though it is explained in \cite{ZemKoszul}*{Remark~5.3} that these contain redundant information and do not appear in the Koszul dual type-$DD$ bimodule.
\end{rem}

\subsubsection{Idempotent $(0,0)$ to $(0,1)$}
\label{sec:horizontal map from (0,0) to (0,1)}

We now describe the components of the differential on ${}^{\cK^!} \cX^{\alg}(L)^{\cK}$ which go from idempotent $(0,0)$ to $(0,1)$.   Write $(\sigma),(\tau)\subset \cK$ for the ideals generated by $\sigma$ and $\tau$, respectively. We pick
\[
R^{\sigma}\colon \cC_{s_1}\to \ys_{s_1- \ell/2}\otimes (\sigma)\quad \text{and} \quad R^{\tau}\colon \cC_{s_1}\to \ys_{s_1+\ell/2}\otimes (\tau)
\]
to be the unique maps satisfying the following properties
\begin{enumerate}
\item $R^{\sigma}$ and $R^{\tau}$ are chain maps.
\item $R^{\sigma}$ and $R^{\tau}$ vanish on algebraic grading 1, and are non-vanishing on each monomial generator of $\cC_s$ in algebraic grading 0.
\item The maps $R^{\sigma}$ and $R^{\tau}$ are homogeneously graded with respect to the Alexander grading and to furthermore satisfy
\[
A_1(R^{\sigma})=-\ell/2,\quad A_1(R^{\tau})=\ell/2,
\]
and
\[
A_2(R^{\tau})=A_2(R^{\sigma})=-\ell/2.
\]
 (See Equation~\eqref{eq:Alexander-grading-shifts} our grading conventions). 
\item On each Alexander grading in $\bH(L)$, $R^{\sigma}$ and $R^{\tau}$ are homogeneously graded with respect to $(\gr_{\ws},\gr_{\zs})$. On generators which lie in Alexander grading $(s_1,s_2)$, they have grading $(\gr_{\ws},\gr_{\zs})(R^{\sigma})=(0, 2s_2+\ell)$ and $(\gr_{\ws},\gr_{\zs})(R^{\tau})=(-2s_2+\ell,0)$. 
\end{enumerate}

When defining the grading of the maps $R^{\sigma}$ and $R^{\tau}$, we are viewing an element $\ys_s\otimes U^i T^j\sigma$ as having 
\[
(\gr_{\ws},\gr_{\zs})(\ys_{s_1}\otimes U^iT^j\sigma)=(\gr_{\ws},\gr_{\zs})(\ys_{s_1})-(2i,2i),
\]
and
\[
(A_1,A_2)(\ys_{s_1}\otimes U^i T^j \sigma)=(s_1,j),
\]
and similarly for $\tau$-weighted elements.
 
We use the map $1|R^{\sigma}+1|R^{\tau}$ as the component of the differential from ${}^{\cR_0^!} \cX^{\alg}(L)^{\cR_0}$ to ${}^{\cR_0^!} \cX^{\alg}(L)^{\cR_1}$.

Schematically, we indicate this below:
\[
\begin{tikzcd}[column sep={3.2cm,between origins},row sep=3cm]
\,
&
\cdots
	\ar[r, "z|L_z", bend left=40]
	\ar[dr, "1|R^{\tau}",pos=.7]
& \cC_{s_1}
	\ar[r, "z|L_z", bend left=40]
	\ar[l, "w|L_w", bend left=40]
	\ar[loop above, "\substack{zw|L_{zw}+\\wz|L_{wz}+\\\theta|U}",looseness=20]
	\ar[dl, "1|R^{\sigma}", swap,crossing over, pos=.7]
	\ar[dr, "1|R^{\tau}",pos=.7]
& \cdots
	\ar[l, "w|L_w", bend left=40]
%	\ar[loop above, "\substack{wz|h_{W,Z}+\\zw|h_{Z,W}+\\\theta|U}",looseness=20]
%	\ar[dr, "D^{\tau_2}",pos=.3]
	\ar[dl, "1|R^{\sigma}",swap, pos=.7,crossing over]	
&
	\\
	\cdots
	\ys_{s_1-\ell/2-1}\cdots
		\ar[r, "z|U^{\eta(s_1-\ell/2-1)}", bend left=30]
		\ar[loop below,"\theta|U"]
	& \ys_{s_1-\ell/2}
		\ar[r, "z|U^{\eta(s_1-\ell/2)}", bend left=30]
		\ar[loop below,"\theta|U"]
		\ar[l, "w|U^{1-\eta(s_1-\ell/2-1)}", bend left=30]
	& \cdots
		\ar[r, bend left=30, "z|U^{\eta(s_1+\ell/2)}"]
		\ar[l, bend left=30, "w|U^{\eta(s_1-\ell/2)}",pos=.55]
		%\ar[loop below,"\theta|U"]
	& \ys_{s_1+\ell/2}
		\ar[r, "z|U^{\eta(s_1+\ell/2)}", bend left=30]
		\ar[l, "w|U^{1-\eta(s_1+\ell/2)}", bend left=30,pos=.45]
		\ar[loop below,"\theta|U"]
	& \ys_{s_1+\ell/2+1}\cdots
		\ar[l, "w|U^{1-\eta(s_1+\ell/2+1)}", bend left=30]
		\ar[loop below,"\theta|U"]
\end{tikzcd}
\]

The following is straightforward, and we leave the proof to the reader:

\begin{lem} The maps $R^{\sigma}$ and $R^{\tau}$ are chain maps of type-$DD$ modules from ${}^{\cR_0} \cX^{\alg}(L)^{\cR_0}$ to ${}^{\cR_0} \cX^{\alg}(L)^{\cR_1}$. 
\end{lem}

Note that in the above, when we say that $R^{\sigma}$ is a chain map from ${}^{\cR_0} \cX^{\alg}(L)^{\cR_0}$ to ${}^{\cR_0}\cX^{\alg}(L)^{\cR_1}$, we are viewing both $\cR_0$ and $\cR_1$ as being inside of $\cK$, and therefore viewing both modules as $DD$ bimodules over $(\cR_0,\cK)$. Furthermore, when we call these maps chain maps, we mean that the morphism differential vanishes on them.

\subsubsection{Idempotent $(1,0)$ to $(1,1)$}
\label{sec:horizontal map from (1,0) to (1,1)}
The staircase complex $\cS=\cCFK(K_2)$ extends to a type-$D$ structure $\cX_0(K_2)^{\cK}$, which is normalized so that $\cX_0(K_2)\cdot \ve{I}_1$ consists of a single generator, concentrated in $(\gr,A)$ grading $0$. See \cite{HMSZ-Naturality}*{Section~5.3} for more details. Let $\delta^{\sigma}$ and $\delta^{\tau}$ denote the components of the structure map on $\cX_0(K_2)^{\cK}$ which are weighted by $\sigma$ and $\tau$.

 We define maps $\Psi^{K_2}$ and $\Psi^{-K_2}$ from 
$ {}^{\cR_1^!} \cX^{\alg}(L)^{\cR_0}$ to ${}^{\cR_1^!} \cX^{\alg}(L)^{\cR_1}$ 
  via the formulas
\[
\Psi^{K_2}=1\otimes L_{T^0}\otimes \delta^{\sigma}\quad \text{and} \quad \Psi^{-K_2}=1\otimes L_{T^{\ell}}\otimes \delta^{\tau}.
\]
Here $L_{T^i}$ is the endomorphism of $\bF[T,T^{-1}]$ which is multiplication by $T^i$. We define the components of the differential $\delta^{1,1}$ of ${}^{\cK^!} \cX^{\alg}(L)^{\cK}$ which go from $ {}^{\cR_1^!} \cX^{\alg}(L)^{\cR_0}$ to ${}^{\cR_1^!} \cX^{\alg}(L)^{\cR_1}$ to be the sum of $\Psi^{K_2}$ and $\Psi^{-K_2}$. 

\subsubsection{Idempotent $(0,1)$ to $(1,1)$}
\label{sec:vertical map from (0,1) to (1,1)}
Finally, we define the component of $\delta^{1,1}$ which goes from ${}^{\cR_0^!}\cX^{\alg}(L)^{\cR_1}$ to ${}^{\cR_1^!}\cX^{\alg}(L)^{\cR_1}$. We recall that the underlying vector spaces of ${}^{\cR_0^!}\cX^{\alg}(L)^{\cR_1}$ and ${}^{\cR_1^!}\cX^{\alg}(L)^{\cR_1}$ are identified  with the idempotent 0 and 1 subspaces of the type-$D$-structure 
\[
{}^{\cK^!} \cX_0(K_1)^{\bF[U]}={}^{\cK^!} [\cCo]^{\cK}\boxtimes {}_{\cK} \cX_0(K_1)^{\bF[U]}
\] (which as vector spaces are the same as the idempotent $0$ and $1$ subspaces of ${}_{\cK} \cX_0(K_1)^{\bF[U]}$).  Write $L_\sigma,L_\tau\colon \ve{I}_0\cdot \cX_0(K_1)\to \ve{I}_1\cdot \cX_0(K_1)$ for the structure maps $\delta_2^1(\sigma,-)$ and $\delta_2^1(\tau,-)$. Then the differentials from ${}^{\cR_0^!}\cX^{\alg}(L)^{\cR_1}$ to ${}^{\cR_1^!}\cX^{\alg}(L)^{\cR_1}$ are given by $s\otimes L_\sigma$ and $t\otimes T^{\ell} L_\tau$, as shown in the diagram below:
\[
\begin{tikzcd}[column sep=2cm, row sep=3cm]
\cdots
\ys_{-2}\cdots
	\ar[r, "z|U^{\eta(-2)}", bend left=30]
	\ar[loop above,"\theta|U"]
	\ar[d, "\substack{s|L_\sigma+\\ t| T^{\ell} L_\tau}"]
& \ys_{-1}
	\ar[r, "z|U^{\eta(-1)}", bend left=30]
	\ar[l, "w|U^{1-\eta(-2)}", bend left=30]
	\ar[loop above,"\theta|U"]
		\ar[d, "\substack{s|L_\sigma+\\ t| T^{\ell} L_\tau}"]
& \ys_0
	\ar[r, "z|U^{\eta(0)}", bend left=30]
	\ar[l, "w|U^{1-\eta(-1)}", bend left=30]
	\ar[loop above,"\theta|U"]
	\ar[d, "\substack{s|L_\sigma+\\ t| T^{\ell} L_\tau}"]
& \ys_1
	\ar[r, "z|U^{\eta(1)}", bend left=30]
	\ar[l, "w|U^{1-\eta(0)}", bend left=30]
	\ar[loop above,"\theta|U"]
	\ar[d, "\substack{s|L_\sigma+\\ t| T^{\ell} L_\tau}"]
& \ys_2\cdots
	\ar[l, "w|U^{1-\eta(1)}", bend left=30]
	\ar[loop above,"\theta|U"]
	\ar[d, "\substack{s|L_\sigma+\\ t| T^{\ell} L_\tau}"]
	\\
\cdots
T^{-2}\ve{e} 
	\ar[r, "\varphi_+|1", bend left=30]
	\ar[loop below,"\theta|U"]
& T^{-1}\ve{e}
	\ar[r, "\varphi_+|1", bend left=30]
	\ar[l, "\varphi_-|1", bend left=30]
	\ar[loop below,"\theta|U"]
& T^0\ve{e}
	\ar[r, "\varphi_+|1", bend left=30]
	\ar[l, "\varphi_-|1", bend left=30]
	\ar[loop below,"\theta|U"]
& T^{1}\ve{e}
	\ar[r, "\varphi_+|1", bend left=30]
	\ar[l, "\varphi_-|1", bend left=30]
	\ar[loop below,"\theta|U"]
& T^{2} \ve{e}
	\ar[l, "\varphi_+|1", bend left=30]
	\ar[loop below,"\theta|U"]
\end{tikzcd}
\]
\begin{rem} In the above diagram, the $T^\ell $ weight on the arrows $t|T^\ell L_\tau$ is part of the $\cR_1$ output of the map. This is not the same as the $T$-powers shown along the bottom row, which are part of the vector space of the module ${}^{\cR_1^!} \cX^{\alg}(L)^{\cR_1}$. 
\end{rem}

\subsection{The length 2 maps} 

We define the complex ${}^{\cK^!}\cX^{\alg}(L)^{\cK}$ to have vanishing length 2 map. (By length 2 map, we mean the component of $\delta^{1,1}$ which goes from idempotent $(0,0)$ to $(1,1)$).

Having defined the structure map $\delta^{1,1}$ on ${}^{\cK^!} \cX^{\alg}(L)^{\cK}$, we state the following:

\begin{prop}
 The bimodule ${}^{\cK^!} \cX^{\alg}(L)^{\cK}$ satisfies the type-$DD$ structure relation, and furthermore is well-defined up to $U$-equivariant homotopy equivalence.
\end{prop} 

We leave the proof to the reader. Compare \cite{CZZSatellites}*{Proposition~5.10}.

\section{Proof of the equivalence}

In this section, we prove the following theorem:

\begin{prop}\label{prop:equivalence-completed-algebra} If $L$ is a 2-component L-space link in $S^3$, and $\scA$ is an arc-system such that ${}^{\ve{\cK}^!} \cX(L,\scA)^{\cK}$ admits a $U$-equivariant model, then there is a homotopy equivalence
\[
{}^{\ve{\cK}^!} \cX^{\alg}(L)^{\cK}\simeq {}^{\ve{\cK}^!} \cX(L,\scA)^{\cK}.
\]
\end{prop}

Recall that $\ve{\cK}^!$ is obtained by completing $\cK^!$ by the weight filtration on idempotent $(0,0)$. That is, we allow infinite sums of monomials in idempotent $(0,0)$ with increasing weight, such as $\sum_{i=0}^\infty (wz)^i$. In the subsequent Section~\ref{sec:uncompleted-algebra}, we prove that the homotopy equivalence holds also over the uncompleted algebra $\cK^!$. 

\begin{rem} In our proof of Proposition~\ref{prop:equivalence-completed-algebra}, we the $U$-equivariant assumption is only necessary to determine the length 2 map of the complex (i.e. the component of $\delta^{1,1}$ which shifts both idempotents). We note that our computation of ${}_{\cK} \cX(L)^{\cR_0}$ in \cite{CZZSatellites} never used the $U$-equivariance property. In our present paper, we assume the $U$-equivariance condition even where it is not necessary, because it simplifies our proofs. Nonetheless, the techniques of this paper extend easily to handle the cases where the $U$-equivariance condition is unncessary, such as for the bimodule ${}_{\cK} \cX(L)^{\cR_0}$. We leave these details as an exercise to the interested reader.
\end{rem}

\subsection{L-space knots and staircases}

In this section, we recall \cite{CZZSatellites}*{Theorem~5.14}:

\begin{lem}
\label{lem:free-resolution-homology} Suppose that $\cC^{\cR_0}$ is a type-$D$ structure admitting a $(\gr_{\ws},\gr_{\zs})$-bigrading whose homology module $\cH_{\cR_0}$ (i.e. the homology of $\cC^{\cR_0}\boxtimes \cR_0$ equipped with only $m_2$ non-vanishing) has support contained in $2\Z\times 2\Z$. Then $\cC^{\cR_0}$ is homotopy equivalent to a free resolution of $\cH_{\cR_0}$.
\end{lem}
\begin{proof}
By definition, $\cH_{\cR_0}$ is quasi-isomorphic to a free resolution of itself. If $A$ is an algebra over a field, then in the category of $A_\infty$-modules over $A$,  quasi-isomorphisms admit inverses as $A_\infty$ module maps (see \cite{KellerNotes}*{Theorem, pg 14}), so $\cH_{\cR_0}$ is homotopy equivalent to a free resolution of itself as an $A_\infty$-module. 

 By homological perturbation theory, we can equip the underlying vector space of $\cH{}_{\cR_0}$ with an $A_\infty$-module structure, denoted $\tilde{\cH}_{\cR_0}$, so that 
\[
\tilde{\cH}_{\cR_0}\simeq \cC^{\cR_0}\boxtimes {}_{\cR_0} [\cR_0]{}_{\cR_0}.
\]
The bimodule ${}_{\cR_0} [\cR_0]_{\cR_0}$ is quasi-invertible, with quasi-inverse equal to the rank four type-$DD$ bimodule ${}^{\cR_0} \Lambda^{\cR_0}$ shown below:
\[
{}^{\cR_0} \Lambda^{\cR_0}:=
\begin{tikzcd}[column sep=1.3cm, row sep=1.3cm]
1
	\ar[r, "W\otimes 1+1\otimes W"]
	\ar[d, "Z\otimes 1+1\otimes  Z",swap]
&
\phi
	\ar[d, "Z\otimes 1+1\otimes Z"]\\
\psi
	\ar[r, "W\otimes 1+1\otimes W",swap] 
& 
\phi\psi.
\end{tikzcd}
\]
The above is well known, and essentially the Koszul resolution of a module over $\cR_0$, but we refer the reader to \cite{CZZSatellites}*{Lemma~5.11} for a proof in our present notation.
We observe that the action $m_{3}(-,a,b)$ from $\tilde{\cH}_{R_0}$ shifts both of the $\gr_{\ws}$ and $\gr_{\zs}$ grading by 1, modulo 2, for any $a,b\in \cR_0$. This implies that $m_{3}(-,a,b)$ vanishes on $\tilde{\cH}_{\cR_0}$ because $\cH$ is supported only in even $\gr_{\ws}$ and $\gr_{\zs}$ gradings. In particular, 
\[
\tilde{\cH}_{\cR_0}\boxtimes {}^{\cR_0} \Lambda^{\cR_0}\simeq \cH_{\cR_0} \boxtimes {}^{\cR_0}\Lambda^{\cR_0},
\]
from which we conclude that $\cH_{\cR_0}\simeq \tilde{\cH}_{\cR_0}.$
\end{proof}

\begin{rem} The above lemma gives another proof (or at least another perspective on the proof) of Ozsv\'{a}th and Szab\'{o}'s result \cite{OSlens} that the knot Floer complexes of L-space knots are homotopy equivalent to staircase complexes. Indeed, if $K$ is an L-space knot, its homology $\cHFK(K)$ can naturally be viewed as a monomial ideal in $\bF[W,Z]$, and is supported in even $\gr_{\ws}$ and $\gr_{\zs}$ grading. The above lemma implies that $\cCFK(K)$ is a free resolution of $\cHFK(K)$, which is a staircase complex. 
\end{rem}

\subsection{The complexes at each idempotent}
 
 In this section, we will prove the following:
 
 \begin{prop}\label{prop:equivalence-over-completed-algebra} Suppose that $L\subset S^3$ is a 2-component L-space link and that  ${}^{\cK^!} \cX(L)^{\cK}$ is a $U$-equivariant model of the link surgery complex of $L$. Write ${}^{\cR_{\veps}} \cX(L)^{\cR_{\veps'}}$ for the complex at idempotent $(\veps,\veps')$. Then there is a $U$-equivariant homotopy equivalence
 \[
 {}^{\cR_{\veps}} \cX(L)^{\cR_{\veps'}}\simeq {}^{\cR_{\veps}} \cX^{\alg}(L)^{\cR_{\veps'}}.
 \] When $\veps=0$, the equivalence is over the completed ring $\ve{\cR}_{0}^!$.
 \end{prop}
 
 Proposition~\ref{prop:equivalence-completed-algebra} follows from Lemmas~\ref{lem:cube-point-00}, \ref{lem:cube-point-10} and \ref{lem:cube-points-01-11}, below.

\subsubsection{Idempotent $(0,0)$}

For each $s_1\in \Z+\lk(K_1,K_2)/2$, we write $X_{s_1}$ and $\cC_{s_1}$ for subspaces of ${}^{\cR_0^!}\cX(L)^{\cR_0}$ and ${}^{\cR_0^!}\cX^{\alg}(L)^{\cR_0}$ which have $A_1$-Alexander grading $s_1$. We view $X_{s_1}$ and $\cC_{s_1}$ as type-$D$ structures over $\cR_0$ by equipping them with the component of $\delta^{1,1}$ which is weighted by $1\in \cR_0^!$. 
\begin{lem}\label{lem:SDR-Cs1-Xs1} For all $s_1$, there is a strong deformation retraction
\[
\begin{tikzcd}
X_{s_1}\ar[r, "\Pi", shift left] \ar[loop left, "H"] & \cC_{s_1} \ar[l, "I",shift left].
\end{tikzcd}
\]
We recall that a \emph{strong deformation retraction} is a diagram as above so that 
\[\Pi \circ I=\bI_{\cC_{s_1}}, \quad I\circ \Pi=\bI_{X_{s_1}}+\d(H),\quad H\circ H=0, \quad H\circ I=0, \quad \Pi\circ H=0.
\]
\end{lem}
\begin{proof}
The homology of $X_{s_1}$ and $\cC_{s_1}$ are both isomorphic to
\begin{equation}
\bigoplus_{s_2\in \Z+\lk(K_1,K_2)/2}\cHFL(L,s_1,s_2). \label{eq:C-s-homology}
\end{equation}
Since $L$ is an L-space link, the group appearing in Equation~\eqref{eq:C-s-homology} is supported only in even $(\gr_{\ws},\gr_{\zs})$ gradings. By Lemma~\ref{lem:free-resolution-homology}, $X_{s_1}$ must be homotopy equivalent to a staircase complex.  We have, therefore, a diagram
\[
\begin{tikzcd}
X_{s_1}\ar[r, "\Pi", shift left] \ar[loop left, "Q"] & \cC_{s_1} \ar[l, "I",shift left] \ar[loop right, "J"]
\end{tikzcd}
\]
where $\Pi\circ I=\bI_{X_s}+\d(J)$ and $I\circ \Pi=\bI_{\cC_s}+\d(Q)$. As a first step, we observe that since $\cC_{s_1}$ is reduced (i.e. the differential has image in the maximal ideal $(W,Z)$), the homotopy equivalence $\Pi \circ I\colon \cC_{s_1}\to \cC_{s_1}$ must be invertible as a type-$D$ module map (alternatively, we could use Lemma~\ref{lem:grading-preserving-endo=id} which is particular to staircases). By composing $\Pi$ with the inverse of this map, we may assume that $\Pi\circ I=\bI_{\cC_{s_1}}$.

The following formula of Lambe and Stasheff \cite{LambeStasheff}*{Section~2.1} gives a homotopy $H\colon X_{s_1}\to X_{s_1}$ which turns the above diagram into a strong deformation retraction:
\[
\begin{split} Q'&=(\bI+I\circ \Pi)\circ Q\circ (\bI+I\circ \Pi)\\
H&=Q'\circ \d_{X_{s_1}}\circ Q',
\end{split}
\]
completing the proof.
\end{proof}

\begin{lem}
\label{lem:cube-point-00} If $L$ is a 2-component L-space link and ${}^{\bm{\cR}_0^!} \cX(L)^{\cR_0}$ is $U$-equivariant, then there is a $U$-equivariant homotopy equivalence of $DD$-bimodules ${}^{\bm{\cR}_0^!} \cX(L)^{\cR_0}\to {}^{\bm{\cR}_0^!} \cX^{\alg}(L)^{\cR_0}.$ 
\end{lem}

\begin{proof} We will write the structure map ${}^{\ve{\cR}_0^!}\cX(L)^{\cR_0}$ as
\[
\delta^{1,1}=d^{1,1}+\a^{1,1},
\]
where $d^{1,1}$ consists only of the differentials weighted by $1\in \cR_0^!$, and $\a^{1,1}$ consists of the components weighted by monomials of $\cR_0^!$ with weight 1 or greater.

We will write $\delta^{1,1}_{\alg}$ also as $d^{1,1}_{\alg}+\a_{\alg}^{1,1}$ where $d^{1,1}_{\alg}$ and $\a_{\alg}^{1,1}$ are defined similarly to the above.  Lemma~\ref{lem:SDR-Cs1-Xs1} gives a strong deformation retraction of type-$D$ modules over $\cR_0$
\[
\begin{tikzcd}\left(\displaystyle \bigoplus_{s_1} X_{s_1},d^{1,1}\right)  \ar[loop left, "H"] \ar[r, "\Pi",shift left] & \left(\displaystyle\bigoplus_{s_1} \cC_{s_1},d^{1,1}_{\alg}\right) \ar[l, "I",shift left]
\end{tikzcd}.
\]
The endomorphism $(1+H\circ \a^{1,1})$ is invertible when we work over $\ve{\cR}_0^!$, since the inverse is given by
\[
\sum_{i=0}^\infty (H\circ \a^{1,1})^n.
\]
The homological perturbation lemma for type-$DD$ modules (compare \cite{LambeStasheff}*{Section~2}, which states the perturbation lemma for chain complexes) gives a homotopy equivalence (in fact, a strong deformation retraction)
\[
\left(\bigoplus_{s_1} X_{s_1},d^{1,1}\right)\simeq \left(\bigoplus_{s_1} \cC_{s_1},d^{1,1}_{\alg}+\b^{1,1}\right) 
\]
where 
\[
\b^{1,1}=\sum_{n \ge 0} \Pi \circ \a^{1,1}\circ (H\circ \a^{1,1})^n \circ I.
\]
Since $H\circ H=0$, it is easy to see that $\b^{1,1}$ is $U$-equivariant (i.e. the only self arrows consist of $\theta\otimes \id\otimes U$), as are the maps appearing in the above homotopy equivalence.

By replacing ${}^{\ve{\cR}_0^!} \cX(L)^{\cR_0}$ with $(\bigoplus_{s_1} \cC_{s_1}, d^{1,1}_{\alg}+\b^{1,1})$,
we can regard ${}^{\ve{\cR}_0^!} \cX(L)^{\cR_0}$ and ${}^{\ve{\cR}_0^!} \cX^{\alg}(L)^{\cR_0}$ as having the same underlying vector space, and also the same induced type-$D$ structures over $\cR_0$. We will write $\cX(L)$ and $\cC_{s_1}$ for the common underlying vector spaces and type-$D$ structures over $\cR_0$. We will write $\delta^{1,1}$ for the structure maps from $\cX(L)$, and we will write $\delta_{\alg}^{1,1}$ for the structure maps from $\cX^{\alg}(L)$.

Next,  we observe that the $w$ and $z$ weighted terms of $\delta^{1,1}$ have the same $(\gr_{\ws},\gr_{\zs})$-gradings as the maps $L_w$ and $L_z$, defined in Section~\ref{sec:complexes-at-idempotents}. Furthermore, they cannot be null-homotopic as chain maps from $\cC_{s_1}$ to $\cC_{s_1\pm 1}$ because by Koszul duality the $w$ weighted component of $\delta^{1,1}$ corresponds to the action of $W_1$ on $\cHFL(L)$, which acts non-trivially on homology because $W_1Z_1=U_1$, and $\cHFL(L)$ is a free and non-zero $\bF[U_1]$-module by definition.  A similar statement holds for the $z$-weighted component. Also, since the $w$ and $z$ weighted components preserve the $(\gr_{\ws},\gr_{\zs})$-gradings modulo 2, it is easy to see that they must preserve the algebraic (staircase) grading on $\cC_{s_1}$. Therefore in the construction of ${}^{\cR_0^!}\cX^{\alg}(L)^{\cR_0}$, we can take $L_w$ and $L_z$ to be $w$ and $z$ components of $\delta^{1,1}$ from $\cX(L)$.

Next, we consider the components of $\delta^{1,1}$ which are weighted by $wz$ or $zw$. Call these $L_{wz}$ and $L_{zw}$, respectively. \emph{A-priori}, $L_{wz}$ and $L_{zw}$ may have summands which increase the algebraic grading by $1$, or decrease it by 1. Write $L_{wz}^{10}, L_{zw}^{10}$ and $L_{wz}^{01}, L_{zw}^{01}$ for each of these summands, respectively.  We now modify $({}^{\cR_0^!}\cX(L)^{\cR_0},\delta^{1,1})$ by a homotopy equivalence so that $L_{wz}$ and $L_{zw}$ increase the algebraic grading, i.e. so that $L_{wz}^{01}$ vanishes. By Lemma~\ref{lem:ext-computation}, the maps $L_{wz}^{01}$ and $L_{zw}^{01}$ are null-homotopic. Write $h_{wz},h_{zw}\colon \cC_{s_1}\to \cC_{s_1}\otimes \cR_0$ for the null-homotopies. We define an endomorphism $h=wz\otimes h_{wz}+zw\otimes h_{zw}$ of $({}^{\cR_0^!} \cX(L)^{\cR_0},\delta^{1,1})$. Note that $1+h$ is invertible over the completed ring $\ve{\cR}_0^!$, so if we apply Lemma~\ref{lem:non-standard-perturbation-lemma}, we see that $({}^{\ve{\cR}_0^!}\cX(L)^{\cR_0},\delta^{1,1})$ is isomorphic to a $DD$-bimodule whose underlying vector space is $\cX(L)$, and whose $w$, $z$, $wz$ and $zw$ weighted differentials coincide with those of $\delta^{1,1}_{\alg}$ (but which may have differentials weighted by algebra elements in $\bm{\cR}_0^!$ with weight larger than 2). Abusing notation slightly, write $\delta^{1,1}$ for the new structure map on ${}^{\cR_0^!} \cX(L)^{\cR_0}$ after this perturbation.

We now eliminate the higher weight differentials of $\delta^{1,1}$, by an inductive procedure which repeats the ideas used to modify the weight 1 and 2 differentials of $\delta^{1,1}$.  Write $L_{wzw}$ and $L_{zwz}$ for the summands of $\delta^{1,1}$ which are weighted by $wzw$ and $zwz$, respectively. We observe that $L_{zwz}$ and $L_{wzw}$ have the same $(\gr_{\ws},\gr_{\zs})$-gradings as $L_w$ and $L_z$. By considering the $wzw$ coefficient of the type-$DD$-structure relation for $\delta^{1,1}$, we see that
\[
\d_{\Int}(L_{wzw})+L_w\circ L_{wz}+L_{zw}\circ L_z=0.
\]
Since $\delta_{\alg}^{1,1}$ satisfies the type-$DD$ structure relation, and $\delta_{\alg}^{1,1}$ has no components with $\cR_0^!$-weight larger than 2, and since $\delta_{\alg}^{1,1}$ coincides with $\delta^{1,1}$ on components with $\cR_0^!$ weight 2 or less, we conclude that $L_{w} \circ L_{wz}+L_{zw}\circ L_z=0$. In particular, $L_{wzw}$ is a chain map when viewed as a type-$D$ morphism from $\cC_{s_1}$ to $\cC_{s_1-1}$. The same logic shows that $L_{zwz}$ is also a chain map from $\cC_{s_1}$ to $\cC_{s_1+1}$.

When we examine the $wzwz$ coefficient of the type-$DD$ structure relation for $\delta^{1,1}$, we see that
\begin{equation}
0\simeq L_z\circ L_{wzw}+L_{zwz}\circ L_w+L_{wz}\circ L_{wz}=L_z\circ L_{wzw}+L_{zwz}\circ L_w, \label{eq:wzwz-relation}
\end{equation}
as type-$D$ endomorphisms of $\cC_{s_1}$. (The homotopy is given by $L_{wzwz}$ and we are also using the fact that $L_{wz}\circ L_{wz}=0$ because $L_{wz}$ increases the algebraic grading by 1). Recall that the maps $L_w$ and $L_z$ are, up to chain homotopy, the unique chain maps from $\cC_{s_1}$ to $\cC_{s_1\pm 1}$ in their grading which are non-zero on homology.  Equation~\eqref{eq:wzwz-relation} implies that, for each $s_1$, the maps
\[
L_{zwz}\colon \cC_{s_1-1}\to \cC_{s_1}\quad \text{and} \quad L_{wzw}\colon \cC_{s_1}\to \cC_{s_1-1}
\]
are either both non-zero on homology, or both zero on homology (though the property of being zero or non-zero may depend on $s_1$). Define a function
\[
\epsilon\colon \Z+\lk(K_1,K_2)/2\to \Z/2
\]
as follows. We set $\epsilon(s_1)=1$ if and only $L_{zwz}\colon \cC_{s_1-1}\to \cC_{s_1}$ and $L_{wzw}\colon \cC_{s_1}\to \cC_{s_1-1}$ are non-zero on homology. Then we define a map 
\[
h\colon {}^{\ve{\cR}_0^!} \cX(L)^{\cR_0}\to {}^{\ve{\cR}_0^!} \cX(L)^{\cR_0}
\] 
which sends $\xs\in \cC_{s_1}$ to $\epsilon(s_1) wz\otimes \xs$. Observe that $1+h$ is invertible over the completed algebra $\ve{\cR}_0^!$, so we can apply Lemma~\ref{lem:non-standard-perturbation-lemma} using the function $h$ described above, to modify $({}^{\ve{\cR}_0^!}\cX(L),\delta^{1,1})$ to an isomorphic type-$DD$ structure whose $wzw$ terms change to $L_{wzw}+\epsilon(s_1)L_w$ and whose $zwz$ terms change to $L_{zwz}+\epsilon(s_1+1) L_z$. The components weighted by $\cR_0^!$-elements of weight less than 3 are unchanged. After this modification, all of the maps $h_{zwz}$ and $h_{wzw}$ are still chain maps (when viewed as maps from $\cC_{s_1}$ to $\cC_{s_1\pm 1}$), however they are all null-homotopic. Abusing notation, we write $({}^{\ve{\cR}_0^!} \cX(L)^{\cR_0},\delta^{1,1})$ for the complex with this modified structure map.

 We let $j_{zwz}\colon \cC_{s_1}\to \cC_{s_1+1}$ and $j_{wzw}\colon \cC_{s_1}\to \cC_{s_1-1}$ be suitably graded null-homotopies. We define a map 
 \[
 h'\colon {}^{\ve{\cR}_0^!}\cX(L)^{\cR_0}\to {}^{\ve{\cR}_0^!}\cX(L)^{\cR_0}
 \] 
 by sending $\xs\in \cC_{s_1}$ to $wzw\otimes j_{zwz}(\xs)+ zwz\otimes j_{wzw}(\xs)$.  After modifying $({}^{\ve{\cR}_0} \cX(L)^{\cR_0},\delta^{1,1})$ using Lemma~\ref{lem:non-standard-perturbation-lemma} with the function $h'$, we have produced a homotopy equivalent model whose differential coincides with our candidate model $\delta^{1,1}_{\alg}$ on monomials of $\cR_0^!$ up to $\cR_0^!$-weight 3. Furthermore, this model and the homotopy equivalence are easily seen to be $U$-equivariant. We again will write $({}^{\ve{\cR}_0} \cX(L)^{\cR_0},\delta^{1,1})$ for the new complex.
 
A similar argument works to eliminate components of $\delta^{1,1}$ which are weighted by $wzwz$ or $zwzw$. Write 
 \[
 L_{wzwz},L_{zwzw}\colon \cC_{s_1}\to \cC_{s_1}
 \]
 for the components of $\delta^{1,1}$ which are weighted by $wzwz$ and $zwzw$, respectively. Similar to before, it is straightforward to see that $L_{wzwz}$ and $L_{zwzw}$ are chain maps. Furthermore, we observe that they both have $(\gr_{\ws},\gr_{\zs})$ degree $(-1,-1)$. Write $\cC_{s_1}=(\d_{s_1}\colon \cC^1_{s_1}\to \cC_{s_1}^0)$. The map $L_{wzwz}$ decomposes as a sum $L_{wzwz}^{10}+L_{zwzw}^{01}$, where $L_{wzwz}^{ij}$ maps algebraic grading $j$ to algebraic grading $i$. Since $\d_{s_1}$ is injective, we see that $L_{wzwz}^{10}$ must vanish. The map $L_{zwzw}^{01}$ is null-homotopic by Lemma~\ref{lem:ext-computation}. By a similar argument, $L_{wzwz}$ is also null-homotopic. Therefore, we may remove the differentials weighted by  $wzwz$ and $zwzw$ by applying Lemma~\ref{lem:non-standard-perturbation-lemma}.
 
Proceeding inductively, if we construct a model for $({}^{\ve{\cR}_0^!}\cX(L)^{\cR_0},\delta^{1,1})$ so that $\delta^{1,1}$ and $\delta_{\alg}^{1,1}$ agree up to $\cR_0^!$-weight $n$, for some $n\ge 4$, we now modify $({}^{\ve{\cR}_0^!}\cX(L)^{\cR_0},\delta^{1,1})$ using Lemma~\ref{lem:non-standard-perturbation-lemma} so that $\delta^{1,1}$ agrees with $\delta_{\alg}^{1,1}$ up to monomials of $\cR_0^!$-weight $n+1$. If $n+1$ is even, we may use the same argument as for removing the $wzwz$ and $zwzw$ weighted terms. If $n+1$ is odd, we may use the same argument as for removing the $wzw$ and $zwz$ weighted terms.

This produces an infinite sequence of maps
\[
f_n\colon \cX(L)\to \bm{\cR}_0^!\otimes \cX(L)\otimes \cR_0
\]
which are isomorphisms of $DD$-modules (for suitable structure maps $\delta^{1,1}$). The maps $f_n$ are all equal to $\id+f_n'$, where $f_n'$ involves only monomials of $\bm{\cR}_0^!$ of length at least $n$. Therefore the infinite composition $\cdots f_3\circ f_2\circ f_1$ is well-defined, and gives homotopy equivalence between our original type-$D$ structure ${}^{\bm{\cR}_0^!}\cX(L)^{\cR_0}$ and ${}^{\bm{\cR}_0^!} \cX_{\alg}(L)^{\cR_0}$. As described earlier, all of these maps are $U$-equivariant, and it is easy to see that their inverses are as.
\end{proof}

\subsubsection{Idempotent $(1,0)$}

We now consider idempotent $(1,0)$, i.e. $\ve{I}_1\cdot \cX(L)\cdot \ve{I}_0$.

\begin{lem}\label{lem:cube-point-10} Suppose that ${}^{\cR_1^!} \cX(L)^{\cR_0}$ has $U$-equivariant structure maps. Then there is a $U$-equivariant homotopy equivalence ${}^{\cR_1^!} \cX(L)^{\cR_0}\simeq {}^{\cR_1^!} \cX^{\alg}(L)^{\cR_0}$.
\end{lem}
\begin{proof} The proof is similar to the proof of Lemma~\ref{lem:cube-point-00}. Write $X_{s_1}$ and $T^{s_1}\otimes \cS$ for the subspaces of ${}^{\cR_1^!}\cX(L)^{\cR_0}$ and ${}^{\cR_1^!}\cX^{\alg}(L)^{\cR_0}$ whose $A_1$-Alexander grading is $s_1$. By using only the components of the structure map $\delta^{1,1}$ which are weighted by $1\in \cR_1^!$, we may view $X_{s_1}$ and $T^{s_1}\otimes \cS$ as type-$D$ modules over $\cR_0$. Write $\cX(L)^{\cR_0}$ for the direct sum of all of the $X_{s_1}$. By the construction from \cite{ZemExact}*{Section~3},  we see that $\cX(L)^{\cR_0}$ is homotopy equivalent to the knot Floer complex of $K_2$, with twisted coefficients induced by the class $[K_1]\in H_1(S^3)$. Since $H_1(S^3)=0$, twisting by the class of $[K_1]$ yields just the tensor product of $\cCFK(K_2)\iso \cS$ with $\bF[T,T^{-1}]$. Therefore, $\cX(L)^{\cR_0}$ is homotopy equivalent to $\bF[T,T^{-1}] \otimes_{\bF} \cS$.  Since each $T^{s_1}\otimes \cS$ lies in a different $A_1$-Alexander grading, we conclude using similar reasoning to the proof of Lemma~\ref{lem:cube-point-00} that there is a strong deformation retraction of type-$D$ structures from $X_{s_1}$ to $T^{s_1}\otimes \cS$. Furthermore, as in the proof of Lemma~\ref{lem:cube-point-00}, the homotopy equivalence can be taken to be $U$-equivariant.

As in the proof of Lemma~\ref{lem:cube-point-00}, using the above constructed strong deformation retraction, we will view $\cX(L)$ and $\cX^{\alg}(L)$ as having the same underlying vector space, which we denote by $\cX(L)$. We write $\delta^{1,1}$ and $\delta^{1,1}_{\alg}$ for the two structure maps.

We now consider the components of the differentials which are weighted by non-unit elements of $\cR_1^!$. Write $L_{\varphi_+}$ and $L_{\varphi_-}$ for the components of $\delta^{1,1}$ which are weighted by $\varphi_+$ and $\varphi_-$, respectively. The structure relation 
\[
(\delta^{1,1})^2+\mu_0\otimes \id\otimes 1
\] (where $\mu_0=\varphi_+\varphi_-+\varphi_-\varphi_+$) implies that $L_{\varphi_+} \circ L_{\varphi_-}\simeq \id$ and $L_{\varphi_-}\circ L_{\varphi_+}\simeq \id$, as type-$D$ endomorphisms of $X_{s_1}$. In particular, $L_{\varphi_+}$ gives a homotopy equivalence from $T^s\otimes \cS$ to  $T^{s+1} \otimes \cS$, and similarly for $L_{\varphi_-}$. Since $L_{\varphi_+}$ and $L_{\varphi_-}$ both preserve the $(\gr_{\ws},\gr_{\zs})$-bigrading on $\cS$, they must be equal to the identity map by Lemma~\ref{lem:grading-preserving-endo=id}.    

Since ${}^{\cR_1^!} \cX(L)^{\cR_0}$ is assumed to be $U$-equivariant, the only other weight one monomials appearing in $\delta^{1,1}$ consist of the component $\theta\otimes \id\otimes U$ which occurs on each generator.

We now claim that there are no components of $\delta^{1,1}$ which are weighted by any monomials $a\in \cR_1^!$ which have weight greater than 1.  If $a\in \cR_1^!$ is a product of $\varphi_+$ and $\varphi_-$ which contains $n$ factors, then write $L_a$ for the component of $\delta^{1,1}$ which is weighted by $a$. We observe that $L_a$ shifts the $(\gr_{\ws},\gr_{\zs})$-grading by $(n-1,n-1)$. (It has the same grading as $\delta_{n+1}^1(T^{\pm 1},T^{\pm 1},\dots, T^{\pm},-)$ on ${}_{\cK} \cX(L)^{\cK}$). We assume by induction that $\delta^{1,1}$ has no weight $j$ terms for $2\le j\le n-1$. If $a$ has weight $n$, then one verifies easily that $L_a$ is a chain map (when viewed as a map from $\cS$ to $\cS$). Lemma~\ref{lem:staircase-no-maps-positive bidegree} implies that $L_a=0$. Proceeding by induction we obtain a $U$-equivariant homotopy equivalence ${}^{\cR_1^!} \cX(L)^{\cR_0}\simeq {}^{\cR_1^!} \cX^{\alg}(L)^{\cR_0}$, completing the proof.
\end{proof}

\subsubsection{Idempotents $(0,1)$ and $(1,1)$}

We state the following result:

\begin{lem}\label{lem:cube-points-01-11} If ${}^{\ve{\cR}_0^!} \cX(L)^{\cR_1}$ is $U$-equivariant then there is a $U$-equivariant homotopy equivalence ${}^{\ve{\cR}_0^!} \cX(L)^{\cR_1}\simeq {}^{\ve{\cR}_0^!} \cX^{\alg}(L)^{\cR_1}$. The same statement holds for ${}^{\cR_1^!}\cX(L)^{\cR_1}$.
\end{lem}

The proofs of the above statements are not substantially different than the proofs of Lemma~\ref{lem:cube-point-00} and ~\ref{lem:cube-point-10}, so we leave them to the reader.

\subsection{The length 1 differentials}

Using the isomorphisms of the complexes at each idempotent, proven in Lemmas~\ref{lem:cube-point-00}, \ref{lem:cube-point-10} and \ref{lem:cube-points-01-11}, we know that ${}_{\cK}\cX(L)^{\cK}$ is $U$-equivariantly homotopy equivalent to a complex obtained by summing over the complexes of ${}_{\cK} \cX^{\alg}(L)^{\cK}$ at each idempotent, and then adding some idempotent shifting differentials.

We write $\Phi^{\pm K_1}$, $\Phi^{\pm K_2}$, and $\Phi^{\pm K_1,\pm K_2}$ for the idempotent shifting differentials from $\cX(L)$. Here, $\Phi^{K_1}$ denotes the components of $\delta^{1,1}$ which are weighted by a multiple of $s$ (on the left), but not by a multiple of $\sigma$ or $\tau$ (on the right). Similarly, $\Phi^{-K_1}$ denotes the components of $\delta^{1,1}$ which are weighted by a multiple of $t$ on the left, but not by a multiple of $\sigma$ or $\tau$ on the right. The map $\Phi^{K_2}$ denotes the component of $\delta^{1,1}$ which is weighted by a multiple of $\sigma$ on the right, but not $s$ or $t$ on the left. The map $\Phi^{-K_2}$ is similar. The map $\Phi^{K_1,K_2}$ is the component of $\delta^{1,1}$ which is weighted by a multiple of $s$ on the left and $\sigma$ on the right. The other $\Phi^{\pm K_1,\pm K_2}$ are defined similarly.

We write $\Phi_{\alg}^{\pm K_1}$, $\Phi_{\alg}^{\pm K_2}$ and $\Phi_{\alg}^{\pm K_1, \pm K_2}$ for the analogous idempotent shifting differentials from $\cX^{\alg}(L)$. 

We recall, as in Proposition~\ref{prop:equivalence-over-completed-algebra}, that we are working with a $U$-equivariant model for ${}^{\cK^!}\cX(L)^{\cK}$. This assumption simplifies a number of the arguments, but is only essential for the map from idempotent $(0,0)$ to $(1,1)$.

\subsubsection{Differentials from idempotent $(0,0)$ to $(1,0)$}

\begin{lem}\label{lem:PhiK1-standard} The maps $\Phi^{K_1},\Phi_{\alg}^{K_1}\colon {}^{\ve{\cR}_0^!} \cX(L)^{\cR_0}\to {}^{\cR_1^!} \cX(L)^{\cR_0}$ are $U$-equivariantly chain homotopic. Similarly $\Phi^{-K_1}$ and $\Phi_{\alg}^{-K_1}$ are $U$-equivariantly chain homotopic. 
\end{lem}
\begin{rem} In the statement, we are viewing $\ve{\cR}_0^!$ and $\cR_1^!$ as being subalgebras of $\ve{\cK}^!$, and therefore we can view both ${}^{\ve{\cR}_0^!} \cX(L)^{\cR_0}$ and ${}^{\cR_1^!} \cX(L)^{\cR_0}$ as type-$DD$ modules over $\ve{\cK}^!$ and $\cR_0$. The claim is that $\Phi^{K_1}$ and $\Phi^{K_1}_{\alg}$ are chain homotopic as morphisms of type-$DD$ modules, when viewed in this way.
\end{rem} 
\begin{proof}
Write $L_s$ and $L_{t}$ for the components of $\Phi^{K_1}$ and $\Phi^{-K_1}$ which are weighted by $s$ and $t$, respectively. Write $L_s^{\alg}$ and $L_{t}^{\alg}$ for the corresponding components of $\Phi_{\alg}^{K_1}$ and $\Phi_{\alg}^{-K_1}$. We can view these as type-$D$ morphisms from $\cC_{s_1}$ to $\cS$. From the structure relations for $\delta^{1,1}$ and $\delta^{1,1}_{\alg}$, it is straightforward to check that $L_s$, $L_t$, $L_{s}^{\alg}$ and $L_t^{\alg}$ are chain maps.

%Note that we can also view the total map $\Phi^{K_1}$ and $\Psi^{K_1}$ as being a chain map from ${}^{\cR_0^!} \cY(L)^{\cR_0}$ to ${}^{\cR_1^!} \cY(L)^{\cR_0}$ (viewing $\cR_0^!$ and $\cR_1^!$ as being subalgebras of $\cK^!$). 

The maps $L_{s}$ and $L^{\alg}_s$ are chain maps from $\cC_{s_1}$ to $\cS$ which preserve the $(\gr_{\ws},\gr_{\zs})$ grading modulo 2. Since $\Phi^{K_1}$ induces a quasi-isomorphism $\cC_{s_1}\to \cS$ for ${s_1}\gg 0$ (compare \cite{MOIntegerSurgery}*{Lemma~10.1}), it cannot be zero on homology for any $s_1$. By assumption, $\Phi_{\alg}^{K_1}$ has this property, and has the same grading as $\Phi^{K_1}$. By Lemma~\ref{lem:basic-staircase-lemma}, we conclude that $L_s\simeq L_{s}^{\alg}$. Let $h_{s_1}$ be a choice of homotopy. We define a morphism of $DD$-bimodules $h\colon {}^{\ve{\cR}_0^!} \cX(L)^{\cR_0}\to {}^{\cR_1^!} \cX(L)^{\cR_0}$ by sending $\xs\in \cC_{s_1}$ to $s\otimes h_{s_1}(\xs)$. By replacing $\Phi^{K_1}$ with $\Phi^{K_1}+\d(h)$ is a chain map with the same properties, we may assume that $L_s=L_{s}^{\alg}$. (Here, we are writing $\d(h)$ for the type-$DD$ morphism differential applied to $h$). 

We now consider the component of $\Phi^{K_1}$ which is weighted by $zs$. We write $L_{zs}$ for this map. We can decompose this as a sum $L_{zs}^{01}+L_{zs}^{10}$, where $L_{zs}^{ij}$ maps algebraic grading $j$ to algebraic grading $i$.  In our definition of $\cX^{\alg}(L)$, we may take the $zs$ component of the differential to be $L_{zs}^{10}$.

We now show how to eliminate the $L_{zs}^{01}$ term of $\Phi^{K_1}$ via a chain homotopy.

For a fixed value $s_1\in \lk(K_1,K_2)/2+\Z$, we will write $L_{zs;s_1}^{01}\colon \cC_{s_1}\to \cS$ for the above map.
 Firstly, for all large $s_1$, we know that $L_s\colon \cC_{s_1}\to \cS$ is an isomorphism, and also that $L_z\colon \cC_{s_1}\to \cC_{s_1+1}$ is an isomorphism. We have
\[
(\gr_{\ws},\gr_{\zs})(L_{zs;s_1})=(\gr_{\ws},\gr_{\zs})(L_s)+(\gr_{\ws},\gr_{\zs})(L_z)+(1,1).
\]
Identifying $\cC_{s_1}$ and $\cS$ via $L_s\circ L_z$, we may view $L_{zs;s_1}^{01}$ as an endomorphism of $\cS$ which increases the $(\gr_{\ws},\gr_{\zs})$ grading by $(1,1)$. It is not hard to see that $L_{zs;s_1}^{01}$ is a chain map, so Lemma~\ref{lem:staircase-no-maps-positive bidegree} implies that $L_{zs;s_1}^{01}=0$ for $s_1\gg 0$.

In order to eliminate $L_{zs;s_1}^{01}$ for other values of $s_1$, we now define a map $j\colon {}^{\cR_0^!} \cX(L)^{\cR_0}\to {}^{\cR_1^!} \cX(L)^{\cR_0}$. On $\cC_{s_1}$, we set $j$ to be
\[
\sum_{t_1\ge s_1} zs\otimes T^{s_1-\ell/2}\otimes L_{zs,t_1}^{01}\circ (L_z)^{t_1-s_1}.
\]
In detail, this maps $\cC_{s_1}$ to $T^{s_1-\ell/2} \cS$. The sum is finite because we have already arranged for $L_{zs,t_1}^{01}$ to be zero if $t_1\gg 0$. Also, $j$ maps $\cC_{s_1}$ to the same copy of $\cS$ as the map $L_s$. It is straightforward to see that
\[
\d(j)=L_{zs}^{01}.
\]

Adding this homotopy to $\Phi^{K_1}$, we may therefore assume that the $s$ and $zs$ coefficients of $\Phi^{K_1}$ coincide with those of $\Phi_{\alg}^{K_1}$.  Since we assume that ${}^{\cK^!} \cX(L)^{\cK}$ is a $U$-equivariant model, there cannot be any $\theta s$ or $\theta s \varphi_+$ weighted summands, so we conclude that our original $\Phi^{K_1}$ is chain homotopic to $\Phi_{\alg}^{K_1}$.

 The proof for $\Phi^{-K_1}$ is identical, so the proof is complete.
\end{proof}

\subsubsection{Differentials from idempotent $(0,0)$ to $(0,1)$}

\begin{lem}\label{lem:PhiK2-standard} The maps $\Phi^{K_2},\Phi_{\alg}^{K_2}\colon {}^{\bm{\cR}_0^!} \cX(L)^{\cR_0}\to {}^{\bm{\cR}_0^!} \cX(L)^{\cR_1}$ are $U$-equivariantly chain homotopic. Similarly $\Phi^{-K_2}$ and $\Phi_{\alg}^{-K_2}$ are $U$-equivariantly chain homotopic. 
\end{lem}
\begin{proof}We focus on constructing a chain homotopy $\Phi^{K_2}\simeq \Phi_{\alg}^{K_2}$.

We can view the map $\Phi^{K_2}$ as being induced by a map
\[
\Phi^{K_2}\colon {}^{\ve{\cR}_0^!} \cX(L)^{\cR_0}\boxtimes {}_{\cR_0}[\phi^{\sigma}]^{\cR_1}\to {}^{\ve{\cR}_0^!} \cX(L)^{\cR_1}
\]
where $\phi^\sigma\colon \cR_0\to \cR_1$ is the natural algebra morphism described in Equation~\eqref{eq:phi-sig-tau},   and ${}_{\cR_0} [\phi^\sigma]^{\cR_1}$ denotes the induced $DA$-bimodule. The correspondence between $\phi^{K_2}$ and $\Phi^{K_2}$ is as follows. If there is a generator $\xs$ and $\Phi^{K_2}(\xs)$ has a summand $a\otimes \ys\otimes U^i T^j \sigma$, then we view $\phi^{K_2}(\xs)$ as having a component $a\otimes \ys\otimes U^i T^j$.

By construction of the link surgery formula, the map $\phi^{K_2}$ is a homotopy equivalence. Therefore, $\Phi^{K_2}$ and $\Phi_{\alg}^{K_2}$ must differ by post-composition with a grading preserving homotopy automorphism of ${}^{\ve{\cR}_0^!} \cX(L)^{\cR_1}$. It suffices therefore to show that all grading preserving, $U$-equivariant homotopy automorphisms of ${}^{\cR_0^!} \cX_{\alg}(L)^{\cR_1}$ are chain homotopic to the identity. We suppose that $F$ is such a homotopy automorphism. If $a\in \cR_0^!$, we write $F_a$ for the component of $F$ which is weighted by $a$.

In each Alexander grading $s_1\in \Z$, ${}^{\cR_0^!} \cX_{\alg}(L)^{\cR_1}$ contains a single generator, which we will denote $\ys_{s_1}$.

If $F$ is a grading preserving homotopy automorphism, then $F_1$ must be an isomorphism, since there are no $\cR_0^!$-weight 0 components of $\delta^{1,1}$ pointing to or from $\ys_{s_1}$. Consequently, $F_1$ sends $\ys_{s_1}$ to $\ys_{s_1}\otimes 1$ for all $s_1$.

Next, we observe that $F_{z}$ and $F_{w}$ shift the $\gr_{\ws}$-grading by $\pm 1$ and hence must vanish since ${}^{\cR_0^!} \cX(L)^{\cR_1}$ is supported in only even $\gr_{\ws}$-gradings.

We now consider $F_{wz}$ and $F_{zw}$.  By considering the $wzw$ coefficient of the relation  $\d(\Phi^K)=0$, we see that 
\[
F_{wz}\circ L_z+L_z\circ F_{zw}=\d(F_{zwz}).
\]
In particular, the map $F_{zw}$ is non-zero on $\ys_{s_1}$ if and only if $F_{wz}$ is non-zero on $\ys_{s_1+1}$.

 Note that since $L_z\circ L_w=\id\otimes U$, one of $L_z(\ys_{s_1})$ and $L_w(\ys_{s_1+1})$ must be  weighted by $1$ and the other by $U$.  For each $s_1\in \Z$, we will define $E(s_1)\subset \{-1,1\}$, as follows.
 \begin{enumerate}
 \item We declare $-1\in E(s_1)$ if and only if $F_{wz}(\ys_{s_1})$ is non-zero and $L_Z(\ys_{s_1-1})$ is weighted by 1.
 \item We declare $1\in E(s_1)$ if and only if $F_{zw}(\ys_{s_1})$ is non-zero and $L_{W}(\ys_{s_1+1})$ is weighted by $1$. 
\end{enumerate}

Note that since $F$ preserves the $(\gr_{\ws},\gr_{\zs})$-bigrading as well as the Alexander bigrading, the maps $F_{wz}$ and $F_{zw}$ can only have components whose $\cR_1$ factor is $1\in \cR_1$.

We now define a homotopy of $j$ as follows. If $-1\in E(s_1)$, then we declare $j(\ys_{s_1})$ to have a summand of $w\otimes \ys_{s_1-1}\otimes 1$. If $1\in E(s_1)$, then we declare $j(\ys_{s_1})$ to have a summand of $z\otimes \ys_{s_1+1}\otimes 1$. (If $E(s_1)=\{-1,1\}$, we declare $j(\ys_{s})$ to be a sum of the above two elements).

One checks easily that 
\[
\d(j)=wz\otimes F_{wz}+zw\otimes F_{zw}.
\]
Therefore we may assume, after chain homotopy, that $F_{wz}$ and $F_{zw}$ are zero. We illustrate in Figure~\ref{fig:eliminate-wz-zw-j} the situation that $E(s_1)=\{1,-1\}$.

From here, we proceed by induction to remove the summands of $F$ which are weighted by monomials of higher length. By the same grading argument as above, there can never be any summands weighted by monomials of the form $wz\cdots w$ or $zw\cdots z$ (i.e. monomials of odd length). The argument used above to eliminate summands weighted by $wz$ or $zw$ extends without substantial change to eliminate summands weighted by monomials of even length greater than 2 (i.e. $wz\cdots z$ and $zw\cdots w$). Since $F$ is assumed to be $U$-equivariant, there are no components weighted by a multiple of $\theta$. We have thus shown that $F$ is chain homotopic to the identity, so the proof is complete.
\end{proof}

\begin{figure}[h]
\[
\begin{tikzcd}[row sep=2cm, column sep=2.5cm, labels=description]
\cdots\ys_{s_1-1} 
	\ar[r,bend left, "z|1"]
&
\ys_{s_1} 
	\ar[r,bend left, "z|U"]
	\ar[l, bend left, "w|U"]
	\ar[d, "\substack{wz|1\\zw|1}"]
	\ar[dl, dashed, "w|1"]
	\ar[dr, dashed, "z|1"]
&
\ys_{s_1+1}\cdots
	\ar[l, bend left, "w|1"]
\\
\cdots\ys_{s_1-1}
	\ar[r,bend left, "z|1"]
&
\ys_{s_1}
	\ar[r,bend left, "z|U"]
	\ar[l, bend left, "w|U"]
&
\ys_{s_1+1}\cdots
	\ar[l, bend left, "w|1"]
\end{tikzcd}
\]
\caption{The homotopy $j$ (dashed arrows) which eliminates $wz$ and $zw$ weighted arrows of $F$ (solid arrows) when $E(s_1)=\{-1,1\}$.}
\label{fig:eliminate-wz-zw-j}
\end{figure}
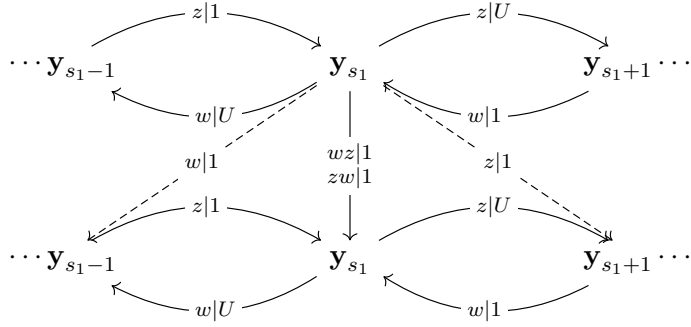

%\begin{rem} The assumption that $F$ is $U$-equivariant is not essential to the above proof. It is straightforward to show that any grading preserving automorphism $F$ of ${}^{\ve{\cR}_0^!}\cY(L)^{\cR_1}$ is chain homotopic to the identity. To prove this, one would first eliminate all of the components of $F$ weighted by $a\in \cR_0^!$ which have length at least 1 and which are not multiples of $\theta$. Next, one would observe that $\theta$ weighted components are prohibited by the grading. To eliminate $\theta w$ and $\theta z$ weighted components, one defines a homotopy $j$ of the form $j(\ys_{s_1})=\epsilon(s_1)\theta \cdot \ys_{s_1}$ where $\epsilon\colon \Z\to \Z/2$ is an appropriately chosen function. To eliminate the higher weighted terms, we observe that grading considerations as above prohibit any components weighted by $\theta a$ where $a$ has even length. The components weighted by $\theta a$ where $a$ has odd length are eliminated using the same strategy as when $a$ has length 1. We leave the details to the interested reader.
%\end{rem}

\subsubsection{Differentials from idempotent $(1,0)$ to $(1,1)$ and from $(0,1)$ to $(1,1)$}

We now state the following lemmas:

\begin{lem}\label{lem:PhiK-Phi-alg-K-extra-cases}\,
\begin{enumerate}
\item The maps $\Phi^{K_2},\Phi_{\alg}^{K_2}\colon {}^{\cR_1^!} \cX(L)^{\cR_0}\to {}^{\cR_1^!} \cX(L)^{\cR_1}$ are $U$-equivariantly chain homotopic. Similarly $\Phi^{-K_2}$ and $\Phi_{\alg}^{-K_2}$ are $U$-equivariantly chain homotopic. 
\item The maps $\Phi^{K_1},\Phi_{\alg}^{K_1}\colon {}^{\bm{\cR}_0^!} \cX(L)^{\cR_1}\to {}^{\cR_1^!} \cX(L)^{\cR_1}$ are $U$-equivariantly chain homotopic. Similarly $\Phi^{-K_1}$ and $\Phi_{\alg}^{-K_1}$ are $U$-equivariantly chain homotopic. 
\end{enumerate}
\end{lem}
The proof of the above lemma follows from identical reasoning as the proofs of Lemma~\ref{lem:PhiK1-standard} and ~\ref{lem:PhiK2-standard}, so we leave the details to the reader.

\subsection{The length 2 differentials}

We now analyze the length 2 differentials. The length 2 differential of the surgery complex decomposes as a sum of four maps, one for each orientation on $L$. The map $\Phi^{K_1\cup K_2}$ (positive orientation on both components) increases the $\gr_{\ws}$ grading by $1$. For an arbitrary orientation on $L$, denoted $\vec{L}$, if we let $\ve{p}\subset \{w_1,w_2,z_1,z_2\}$ contain one basepoint from each of $K_1$ and $K_2$, and furthermore contain $z_i$ if and only if $\fro$ is positively oriented on $K_i$ (and contain $z_i$ otherwise), then the map $\Phi^{\vec{L}}$ increases $\gr_{\ve{p}}$ by 1 (see \cite{ZemBordered}*{Lemma~6.4 (1)}). 

\begin{lem}
\label{lem:length-2-homotopic} Let $L$ be a 2-component L-space link and suppose that ${}^{\cK^!} \cX(L)^{\cK}$ is a $U$-equivariant model of the link surgery complex for $L$. If $F\colon {}^{\ve{\cR}_0^!} \cX(L)^{\cR_0}\to {}^{\cR_1^!} \cX(L)^{\cR_1}$ is a $U$-equivariant chain map which is weighted by multiples of $s$ and $\sigma$, and which has $+1$ $\gr_{\ws}$-grading and the same Alexander grading shift as $\Phi^{K_1\cup K_2}$, then $F$ is null-homotopic. The same statement holds for the other orientations on $K_1\cup K_2$, as long as the $\gr_{\ws}$ grading is replaced with the $\gr_{\ps}$ grading, where $\ps\subset\{w_1,z_1,w_2,z_2\}$ is as above.
\end{lem}
\begin{proof} 
As in the proof of Lemma~\ref{lem:PhiK2-standard}, we can interpret the map $F$  in the statement  as a chain map from
\[
{}^{\cR_0^!} \cX(L)^{\cR_0}\boxtimes {}_{\cR_0}[\phi^\sigma]^{\cR_1}\to {}^{\cR_1^!} \cX(L)^{\cR_1}.
\]
The left hand side is homotopy equivalent ($U$-equivariantly) to ${}^{\cR_0^!} \cX(L)^{\cR_1}$, and so we can view the map $F$ instead as a map
\[
F\colon {}^{\cR_0^!} \cX(L)^{\cR_1}\to {}^{\cR_1^!} \cX(L)^{\cR_1}.
\]
Write $\ys_{s_1}$ for the generators of ${}^{\cR_0^!} \cX(L)^{\cR_1}$ and write $T^{s_1}\ve{e}$ for the generators of ${}^{\cR_1^!} \cX(L)^{\cR_1}$. The only possible components of the map $F$ are to send $\ys_{s_1}$ to $zs\otimes T^{s_1+\delta+1}\ve{e}\otimes U^{H_{K_1}(s_1)-1}$ where $\delta$ is some fixed Alexander grading shift. (Components weighted by $s$ are impossible since a shift in the Maslov grading by $+1$ would be impossible since all generators have even $\gr_{\ws}$ grading; components weighted by $\theta s$ or $\theta s \varphi_+$ are impossible because we have assumed that $F$ is $U$-equivariant). Since $\delta$ is fixed, and the module ${}^{\cR_1^!} \cX(L)^{\cR_1}$ has a translation action which shifts this Alexander grading by $1$ but preserves the Maslov grading, we can assume that $\delta=0$, for convenience. 

Write $\epsilon(s_1)\in \Z/2$ for the coefficient of this component in $F$. (I.e. set $\epsilon(s_1)=1$ if and only if $F(\ys_{s_1})$ is non-zero). That is, say that
\begin{equation}
F(\ys_{s_1})=\epsilon(s_1) \cdot zs\otimes T^{s_1+1}\ve{e}\otimes U^{H_{K_1}(s_1)-1}.
\end{equation}
We observe firstly that $\epsilon(s_1)=0$ for $s_1\gg 0$. This is because $H_{K_1}(s_1)=0$ for $s_1\gg 0$, and so the above equation would involve negative powers of $U$, which is prohibited. In fact, we see that $\epsilon(s_1)=0$ if $H_{K_1}(s_1)=0$. 

We now define our homotopy by the equation
\[
j(\ys_{s_1})=  \left(\sum_{t_1\ge s_1}\epsilon(t_1) \right)s\otimes T^{s_1}\ve{e}\otimes U^{H_{K_1}(s_1)-1}.
\]
It is straightforward to see that this homotopy satisfies $\d(j)=F$.

The proof for other orientations on $L$ carries through \emph{mutatis mutandis}.
\end{proof}

\begin{rem} The above lemma is false if we omit ``$U$-equivariant'' from the statement. For example, we may take $\Phi^{K_1}\circ \Phi^{K_2}$ and multiply the $\cK^!$ entries by $\theta$. 
\end{rem}

\subsection{Proof of Proposition~\ref{prop:equivalence-completed-algebra}}

We now summarize our proof that ${}^{\ve{\cK}^!} \cX(L)^{\cK}\simeq {}^{\ve{\cK}^!} \cX^{\alg}(L)^{\cK}$.

\begin{proof}[Proof of Proposition~\ref{prop:equivalence-completed-algebra}]
The complex ${}^{\ve{\cK}^!} \cX(L)^{\cK}$ decomposes as a hypercube of type-$DD$ bimodules:
\begin{equation}
{}^{\ve{\cK}^!} \cX(L)^{\cK}=\begin{tikzcd}[column sep=2.5cm, row sep=2.5cm] {}^{\ve{\cR}_0^!} \cX(L)^{\cR_0}
	\ar[d, "\Phi^{K_1}+\Phi^{-K_1}"]
	\ar[dr, "\Phi^{\pm K_1,\pm K_2}"]
	\ar[r, "\Phi^{K_2}+\Phi^{-K_2}"]
	 & {}^{\ve{\cR}_0^!} \cX(L)^{\cR_1}
	 \ar[d, "\Phi^{K_1}+\Phi^{-K_1}"]\\
{}^{\cR_1^!} \cX(L)^{\cR_0} \ar[r, "\Phi^{K_2}+\Phi^{-K_2}"] & {}^{\cR_1^!} \cX(L)^{\cR_1}
\end{tikzcd}
\label{eq:hypercube-figure-X-L}
\end{equation}
By assumption, all of these complexes and maps are $U$-equivariant. The complex ${}^{\ve{\cK}^!} \cX^{\alg}(L)^{\cK}$ decomposes similarly. 

Proposition~\ref{prop:equivalence-over-completed-algebra} implies that ${}^{\ve{\cK}^!} \cX(L)^{\cK}$ is homotopy equivalent to a complex similar to the one in Equation~\eqref{eq:hypercube-figure-X-L}, but with ${}^{\cR_\veps^!}\cX^{\alg}(L)^{\cR_\nu}$ replacing each ${}^{\cR_{\veps}^!}\cX(L)^{\cR_\nu}$. Since the homotopy equivalences are $U$-equivariant, the resulting hypercube structure maps are also $U$-equivariant. Applying Lemmas~\ref{lem:PhiK1-standard}, ~\ref{lem:PhiK2-standard} and ~\ref{lem:PhiK-Phi-alg-K-extra-cases}, we conclude that $\Phi^{\pm K_i}\simeq \Phi^{\pm K_i}_{\alg}$, so ${}^{\ve{\cK}^!} \cX(L)^{\cK}$ is homotopy equivalent to a complex that agrees with ${}^{\ve{\cK}^!} \cX(L)^{\cK}$, except possibly along its diagonal (length 2) map. Lemma~\ref{lem:length-2-homotopic} implies that any two choices of diagonal map are homotopic, which implies that the two hypercubes are homotopy equivalent, completing the proof.
\end{proof}

\section{Relating the completed and uncompleted algebras}

\label{sec:uncompleted-algebra}

There are several natural ways that we can complete $\cK^!$:
\begin{enumerate}
\item No completions on  the algebra $\cK^!$.
\item Completing $\cK^!$ with respect to the tensor weight filtration on $\cR_0^!$.
\item Completing $\cK^!$ with respect to the tensor weight filtration on $\cR_0^!$ and $\cR_1^!$.
\item Completing $\cK^!$ with respect to the tensor weight filtration on $\cR_1^!.$
\end{enumerate}
We are interested in the first and second ways of completing the algebra. We will write $\cK^!$ and $\ve{\cK}^!$ for these two manners of completing the algebra. We do not have any applications of the third or fourth, so we will not consider it in this paper.

\begin{prop}
\label{prop:equivalence-of-categories} There is an $A_\infty$-equivalence of categories between ${}^{\cK^!}\MOD$ and ${}^{\ve{\cK}^!} \MOD.$
\end{prop}

\begin{rem}
Conceptually, the proof is related to the fact that $\bm{\cR}_0^!$ and $\cR_0^!$ have homology equal to the exterior algebra on two generators. In fact, there is an $A_\infty$-algebra $\cK_\infty^!$ which is quasi-equivalent to both $\cK^!$ and $\ve{\cK}^!$. The algebra $\cK_\infty^!$ has in idempotent $(0,0)$ the exterior algebra on two generators in place of $\cR_0^!$. See \cite{ZemKoszul}*{Lemma~3.3}.
\end{rem}

 \begin{proof} We will define $A_\infty$-algebra morphisms $\phi$ and $\iota$ and consider the induced bimodules ${}^{\cK^!} [\phi]_{\ve{\cK}^!}$ and ${}^{\ve{\cK}^!} [\iota]_{\cK^!}$ (see \cite{LOTBimodules}*{Definition~2.2.48}). We will prove that $[\phi]\boxtimes [\iota]$ and $[\iota]\boxtimes [\phi]$ are both homotopy equivalent to the identity functors. Furthermore, these bimodules and the morphisms between them will be bounded as $DA$-bimodules (in the sense that $\delta_n^1=0$ if $n\gg 0$).  

The morphism $\iota$ is given by the inclusion map $\iota\colon \cK^!\to \ve{\cK}^!$. That is, we have $\iota_1$ is the inclusion and $\iota_j=0$ for $j>1$. 

The morphism $\phi$ is more complicated. We begin by defining the map $\phi_1\colon \ve{\cK}^!\to \cK^!$. We pick an integer $N\ge 4$. (The choice of $N$ is not important for our proof, so we could pick $N=4$). 

 If $a\in \ve{I}_0\cdot \cK^1\cdot \ve{I}_0$ is a monomial, we define $\wt_\theta(a)\in \{0,1\}$ to be $0$ if $a$ is not a multiple of $\theta$, and to be $1$ otherwise. We write $\wt(a)$ for the number of factors of $a$. (E.g. $\wt(wz)=2$, $\wt(\theta wzwz)=5$). We write $\wt_+(a)=\wt(a)+\wt_\theta(a)$. Notice that $\wt_+$ is preserved by $\mu_1$ and $\mu_2$. We set
\[
\phi_1(a)=
\begin{cases} a & \text{ if } \wt_+(a)\le N\\
0& \text{ otherwise}.
\end{cases}
\]
We define $\phi_1$ to be the identity on all monomials in idempotent $(0,1)$ or $(1,1)$. If $a$ and $b$ are monomials in idempotent $(0,0)$, we set
\[
\phi_2(a,b)=\begin{cases} (ab)'\theta & \text{ if } \wt_+(ab)>N, \wt_+(a)\le N, \wt_+(b)\le N\\
0& \text{ otherwise}.
\end{cases}
\]
In the above $(ab)'$ denotes the monomial obtained from $ab$ by reducing $\wt$ by 2 (by removing the last two non-$\theta$ symbols from $ab$) while preserving $\wt_\theta$. For example, $(wzwzwz)'=wzwz$ and $(zwzw\theta)'=zw\theta$.  We set $\phi_2(a,b)$ to be zero if $a$ or $b$ is not in idempotent $(0,0)$.

The compatibility condition is proven as follows. The compatibility condition with one input amounts to the claim that $\phi_1$ is a chain map, which is easy to verify. The compatibility condition on two inputs, $a$ and $b$, is proven as follows. The compatibility condition amounts to the claim that
\[
\phi_1(ab)+\phi_1(a)\phi_1(b)+\mu_1 \phi_2(a,b)+\phi_2(\mu_1 a,b)+\phi_2(a,\mu_1 b)=0.
\]
Note that the claim is only non-trivial when $ab\neq 0$. 
The proof of this relation breaks into several cases:
\begin{enumerate}
\item ($\wt_+(a)\le N$, $\wt_+(b)\le N$ and $\wt_+(ab)\le N$) In this case, $\phi_2(a,b)=\phi_2(\mu_1 a,b)=\phi_2(a,\mu_1 b)=0$ and the compatibility condition is straightforward to verify.
\item ($\wt_+(a)>N$ or $\wt_+(b)>N$) In this case, $\phi_1(ab)+\phi_1(a)\phi_1(b)=0$, and also $\phi_2(a,b)=\phi_2(\mu_1 a,b)=\phi_2(a,\mu_1 b)=0$, so the compatibility condition follows.
\item ($\wt_+(a)\le N$, $\wt_+(b)\le N$ and $\wt_+(ab)>N$) In this case, $\phi_1(a)\phi_1(b)+\phi_1(ab)=ab$.  If $\wt_\theta(a)=\wt_\theta(b)=0$, then $\mu_1 \phi_2(a,b)=ab$ while $\phi_2(\mu_1a,b)=\phi_2(a,\mu_1b)=0$, so the compatibility condition follows. If $\wt_\theta(a)=1$ while $\wt_\theta(b)=0$, then $\mu_1 \phi_2(a,b)=\phi_2(a,\mu_1b)=0$ while $\phi_2(\mu_1a,b)=ab$, so the compatibility condition follows. The case that $\wt_\theta(a)=0$ while $\wt_\theta(b)=1$ is similar. Finally, if $\wt_\theta(a)=\wt_\theta(b)=1$, then $ab=\mu_1 \phi_2(a,b)=\phi_2(\mu_1 a,b)=\phi_2(a,\mu_1 b)=0$, so the compatibility condition follows.
\end{enumerate}
We now consider the compatibility condition with three or more inputs. We focus on three inputs since the case of more than three inputs follows from similar reasoning. Similar to the previous cases, the condition is only non-trivial if all inputs are in $\ve{I}_0\cdot \cR_0^!\cdot \ve{I}_0$. Let $F\colon R_0^!\otimes R_0^!\otimes R_0^!\to R_0^!$ be the map
\[
F(a,b,c)=\mu_2(\phi_2(a,b), \phi_1(c))+\mu_2(\phi_1(a), \phi_2(b,c))+\phi_2(\mu_2(a,b),c)+\phi_2(a,\mu_2(b,c)).
\]
Since $\phi_3=0$, this is the sum of all of the terms of the compatibility condition. The compatibility conditions for one and two inputs imply that $F$ is a chain map (where we equip the domain with the differential $\mu_1\otimes \bI^2+\bI\otimes \mu_1\otimes \bI+\bI^2\otimes \mu_1$). Furthermore, $\wt_\theta$ gives a grading on both the domain and the codomain of $F$ which $F$ increases by $1$. Schematically, if we write $(C_3\to C_2\to C_1\to C_0)$ for the domain of $F$ (where the subscript indicates the $\wt_\theta$-grading) and $(E_1\to E_0)$ denotes the codomain, then $F$ induces a map from $C_1$ to $E_0$, and $F$ vanishes on $C_3$, $C_2$ and $C_1$. Since $F$ is a chain map and the differential vanishes on $E_0$, we have that $\mu_1\circ F$ vanishes. However $R_0^!$ has trivial homology in $\wt_\theta$-grading 1, so $\mu_1\colon E_1\to E_0$ is injective, and therefore $F=0$. Therefore the compatibility condition is satisfied. A similar argument shows that all of the compatibility conditions with more inputs also are satisfied.

We now construct morphisms of bimodules
\[
\begin{tikzcd}{}^{\cK^!}[\phi]_{\ve{\cK}^!}\boxtimes {}^{\ve{\cK}^!}[\iota]_{\cK^!}
% \ar[loop left, "h"]
  \ar[r,shift left, "f"]&{}^{\cK^!}\bI_{\cK^!} \ar[l, "g", shift left]
%  \ar[loop right, "j"]
\end{tikzcd}
\]
which satisfy $g\circ f=\bI$ and $f\circ g=\bI$. 

We define $f$ as follows. We set $f_1^1(1)=1\otimes 1$. We set $f_2^1$ to vanish except on idempotent $(0,0)$. We set $f_2^1(1,a)$ to vanish unless $\phi_1(a)=0$ (i.e. $\wt_+(a)>N$) and $\wt_\theta(a)=0$. In this case, we set $f_2^1(1,a)=a'\theta\otimes 1$, where $a'$ again means the monomial obtained by reducing the $\wt(a)$ by 2 while preserving $\wt_\theta(a)$. 

One easily checks that $f$ is a morphism of bimodules. We define the map $g$ using the same formula as $f$. It is easy to check that both $f\circ g$ and $g\circ f$ are equal to the identity.

Next, we construct an isomorphism between ${}^{\ve{\cK}^!}[\iota]_{\cK^!}\boxtimes {}^{\cK^!}[\phi]_{\ve{\cK}^!}$ and ${}^{\ve{\cK}^!} \bI_{\ve{\cK}^!}$. We observe that the structure maps of ${}^{\ve{\cK}^!}[\iota]_{\cK^!}\boxtimes {}^{\cK^!}[\phi]_{\ve{\cK}^!}$ have the same formula as those from ${}^{\cK^!}[\phi]_{\ve{\cK}^!}\boxtimes {}^{\ve{\cK}^!}[\iota]_{\cK^!}$. Furthermore, the bimodule maps $f$ and $g$ we defined above reduce the weight filtration by at most 1, and hence they induce well defined maps over the completed algebra $\ve{\cK}^!$. In particular, we also conclude that ${}^{\ve{\cK}^!}[\iota]_{\cK^!}\boxtimes {}^{\cK^!}[\phi]_{\ve{\cK}^!}$ and ${}^{\ve{\cK}^!} \bI_{\ve{\cK}^!}$ are isomorphic, completing the proof. 
\end{proof}

\begin{example} The complex $\begin{tikzcd}\xs \ar[r, "1+wz"]& \ys
\end{tikzcd}$ is contractible in ${}^{\bm{\cR}_0^!} \MOD$, because we can define a null-homotopy $H$ of the identity by $H(\ys)=\sum_{n=0}^\infty (wz)^n \otimes \xs$. We observe that this complex is also contractible in ${}^{\cR_0^!} \MOD$ because we can instead take $H(\ys)=(1+wz)\otimes \xs+ wz\theta\otimes \ys$, $H(\xs)=wz\theta \otimes \xs$. 
\end{example}

As a consequence, we obtain the main theorem:

\begin{thm} If $\scA$ is an arc system so that ${}^{\cK^!} \cX(L;\scA)^{\cK}$ is $U$-equivariant, then there is a homotopy equivalence of type-$DD$ bimodules ${}^{\cK^!} \cX(L;\scA)^{\cK}\simeq {}^{\cK^!} \cX^{\alg}(L)^{\cK}$.
\end{thm}
\begin{proof} Proposition~\ref{prop:equivalence-completed-algebra} implies that there is an equivalence after over the completed algebra $\ve{\cK}^!$. Proposition~\ref{prop:equivalence-of-categories} now implies that this equivalence holds also on the uncompleted algebras $\cK^!$. 
\end{proof}

By tensoring the above equivalence with the dualizing bimodule ${}_{\cK} [\cTr]_{\cK^!}$, we obtain the statement in Theorem~\ref{thm:main-computation-intro}.

\section{Examples}
In this section, we illustrate the construction of the bimodules ${}^{\cK^!} \cX^{\alg}(L)^{\cK}$ when $L$ is the positively clasped Whitehead link $\Wh_+= K_1\cup K_2$, equipped with framing $(\lambda_1,\lambda_2)=(0,0)$. We first describe the bimodule ${}^{\cK^!} \cX^{\alg}(\Wh_+)^{\cK}$, and then we compute the surgery complex of the knot $W\subset S^1\times S^2$ obtained by performing $0$-surgery to one component of $\Wh_+$.

\subsection{The bimodule ${}^{\cK^!} \cX^{\alg}(\Wh_+)^{\cK}$}

In this section, we compute the bimodule ${}^{\cK^!} \cX^{\alg}(\Wh_+)^{\cK}$. The $(0,0)$ and $(1,0)$ idempotents of ${}_{\cK} \cX^{\alg}(\Wh_+)^{\cK}$ are described in \cite{CZZSatellites}*{Section~6.2} (see also \cite{CZZApplications}*{Section~4.1}). As described in \cite{ZemKoszul}*{Section~5.3}, the description of the $DD$-bimodule over $(\cK^!,\cR_0)$ is essentially equivalent to this perspective.

%
%
% Here, we rewrite the result as an type $DD$ bimodule over $\cK^!$ and $\cK$, with minor changes. The $(1,0)$ and $(1,1)$ idempotents are trivial as well, as $K_1$ is the unknot. The only nontrivial parts are the length $1$ maps $\Phi^{\pm K_1}:  {}^{\cR_0^!} \cX(L)^{\cR_{\epsilon}} \to  {}^{\cR_1^!} \cX(L)^{\cR_{\epsilon}}$ for $\epsilon\in\left\{0,1\right\}$, which we will compute explicitly from the gradings in this example. 
 
 In Figure~\ref{fig:Whitehead link}, we illustrate the $H$-function of $\Wh_+$. The boxed entry represents generators of algebraic grading $0$ in the $(0,0)$ idemponent.

\begin{figure}[h!]
	\adjustbox{scale=0.8}{
		\begin{tabular}{m{4cm}p{8cm}}
			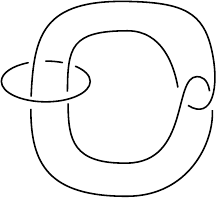
			&$ \begin{array}{|c|cccccccc|}
				\hline
				\hbox{\diagbox{$s_2$}{$s_1$}}&\cdots&-2&-1&0&1&2&\cdots&\\
				\hline	\vdots&\ddots&2&1&0&0&0&\reflectbox{$\ddots$}&\\
				1&\cdots&2&1 &\boxed{0}&0&0 & \cdots& \\
				0&\cdots&\boxed{2}&\boxed{1} &1&\boxed{0}&\boxed{0} &\cdots&\\
				-1&\cdots&3&2 &\boxed{1}&1&1& \cdots&\\
				\vdots&\reflectbox{$\ddots$}&\vdots&\vdots &\vdots&\vdots&\vdots&\ddots &\,\\
				\hline
			\end{array}
			$
		\end{tabular}
	}
	\caption{The $H$-function of the positively clasped Whitehead link $\Wh_+$. }
	\label{fig:Whitehead link}
\end{figure}

The bimodule ${}^{\cR_0^!} \cX(\Wh_+)^{\cR_{0}}$ takes the following form:
	\[	
\adjustbox{scale=0.8}{
	\begin{tikzcd}[labels=description, column sep=1.4cm]
		&&& \xs^0_0
		\ar[dl, bend left =20, "w|Z",pos=0.7]
		\ar[dr, "z|Z", bend left]
		\ar[d,bend left =80, "{zw|Z}"]
		&&
		\\[.5cm]
		\cdots
		\ar[r,bend left, "z|U"]
		& \xs_{-2}
		\ar[r, bend left, "z|U"]
		\ar[l, bend left, "w|1"]
		&
		\xs_{-1}
		\ar[ur, bend left, "z|W"]
		\ar[l, "w|1", bend left]
		&
		\xs_0^1
		\ar[u, "1|W"]
		\ar[d, "1|Z"]
		&
		\xs_1
		\ar[r, bend left, "z|1"]
		\ar[dl, "w|Z", bend left]
		&
		\xs_2
		\ar[l, bend left, "w|U"]
		\ar[r,bend left, "z|1"]
		&
		\cdots
		\ar[l,bend left, "w|U"]
		\\[.5cm]
		&&&
		\xs_0^2
		\ar[ul, "w|W", bend left]
		\ar[ur, "z|W", bend left=20,pos = 0.75]
		\ar[u,bend left =70, "{wz|W}"]
		&&
	\end{tikzcd}
}
\]

Here, $\xs^{j}_{s_1}$ denote the $j$th generator in the staircase complex $\cC_{s_1}$ with the first Alexander grading $A_1=s_1$, and we denote $\xs^{0}_{s_1}$ by $\xs_{s_1}$ if there is only a single generator in $\cC_{s_1}$. The arrows indicate the $\delta^{1,1}$ differential. Each generator has a $\theta$-weighted arrow $\theta|U$ pointing to itself by the $U$-equivariant condition, which we omit from the drawing. (We will omit such arrows for other idempotents as well to simplify the drawing.)

The gradings of the generators are given by 
\[
\begin{split}
(\gr_{\ws},\gr_{\zs},A_1,A_2)(\xs^0_0) &= (0,-2,0,1)\\
 (\gr_{\ws},\gr_{\zs},A_1,A_2)(\xs^1_0) &= (-1,-1,0,0),\\
  (\gr_{\ws},\gr_{\zs},A_1,A_2)(\xs_0^2) &= (-2,0,0,-1),
  \end{split}\]
and for $s_1\neq 0$,
\begin{equation*}(\gr_{\ws},\gr_{\zs},A_1,A_2)(\xs_{s_1}) = \begin{cases}
		(0,-2s_1,s_1,0) & \text{if $s_1>0,$}\\
		(2s_1,0,s_1,0) & \text{if $s_1<0$.}
	\end{cases}
\end{equation*}

Since $K_2$ is an unknot (after deleting $K_1$), the complex $\cS=\cCFK(K_2)$ is spanned by a single generator, for which we write $\xs'$. The bimodule ${}^{\cR_1^!} \cX(\Wh_+)^{\cR_{0}}$ therefore takes the following form: 
	\[	
	\begin{tikzcd}[labels=description, column sep=1.4cm]
		\cdots
		\ar[r, bend left, "\varphi_+|1"]
		& T^{-1}\xs'
		\ar[r, bend left, "\varphi_+|1"]
		\ar[l, bend left, "\varphi_-|1"]
		&T^{0}\xs'
		\ar[r, "\varphi_+|1", bend left]
		\ar[l, "\varphi_-|1", bend left]
		&T^{1}\xs'
		\ar[r, bend left, "\varphi_+|1"]
		\ar[l, bend left,"\varphi_-|1"]
		&\cdots 
		\ar[l, bend left,"\varphi_-|1"]
	\end{tikzcd},
\]
where $\xs'$ is the generator of $\cCFK(K_2)$, and $K_2$ is the unknot. We set
\[
(\gr_{\ws},\gr_{\zs},A_1,A_2)(T^i\xs') = (0,0,i,0).
\]

The $(0,1)$ idempotent ${}^{\cR_0^!} \cX(\Wh_+)^{\cR_{1}}$ takes the following form:
\[	
	\begin{tikzcd}[labels=description, column sep=1.4cm]
		\cdots
		\ar[r, bend left, "z|U"]
		& \ys_{-1}
		\ar[r, bend left, "z|U"]
		\ar[l, bend left, "w|1"]
		&\ys_{0}
		\ar[r, "z|1", bend left]
		\ar[l, "w|1", bend left]
		&\ys_{1}
		\ar[r, bend left, "z|1"]
		\ar[l, bend left,"w|U"]
		&\cdots 
		\ar[l, bend left,"w|U"]
	\end{tikzcd}
\] 
Since $K_1$ is the unknot, we compute \[\eta(s_1)= H_{K_1}(s_1) -H_{K_1}(s_1+1)  = \max\left\{0,-s_1\right\}. \]

The gradings of $\ys_{s_1}$ are given by 

\begin{equation*}(\gr_{\ws},\gr_{\zs},A_1,A_2)(\ys_{s_1}) = \begin{cases}
		(0,-2s_1,s_1,0) & \text{if $s_1\ge0,$}\\
		(2s_1,0,s_1,0) & \text{if $s_1<0$.}
	\end{cases}
\end{equation*}

The $(1,1)$ idempotent ${}^{\cR_1^!} \cX(\Wh_+)^{\cR_{1}}$ takes the following form:

\[	
\begin{tikzcd}[labels=description, column sep=1.4cm]
	\cdots
	\ar[r, bend left, "\varphi_+|1"]
	& T^{-1}\ve{e}
	\ar[r, bend left, "\varphi_+|1"]
	\ar[l, bend left, "\varphi_-|1"]
	&T^{0} \ve{e}
	\ar[r, "\varphi_+|1", bend left]
	\ar[l, "\varphi_-|1", bend left]
	&T^{1}\ve{e}
	\ar[r, bend left, "\varphi_+|1"]
	\ar[l, bend left,"\varphi_-|1"]
	&\cdots 
	\ar[l, bend left,"\varphi_-|1"]
\end{tikzcd},
\]
which is the same for every two-component $L$-space link in $S^3$. Each generator $T^i\ve{e}$ has a single Maslov grading equal to $0$, and has Alexander gradings $(A_1,A_2)(T^i \ve{e})=(i,0)$.

The length $1$ map $\Phi^{+K_1} +\Phi^{-K_1}:  {}^{\cR_0^!} \cX(\Wh_+)^{\cR_0} \to  {}^{\cR_1^!} \cX(\Wh_+)^{\cR_0}$ is shown below (this computation is also performed in \cite{CZZSatellites}*{Section~6.2})
	\[	
\adjustbox{scale=0.8}{
	\begin{tikzcd}[labels=description] 
			\\[1.2cm]
		{}^{\cR_0^!} \cX(\Wh_+)^{\cR_0}\ar[dd, "\Phi^{+K_1} +\Phi^{-K_1} \quad ="]
		\\  [.5cm]
		\\ [2cm]		
		{}^{\cR_1^!} \cX(\Wh_+)^{\cR_0}
	\end{tikzcd}
\hspace{-.2cm}
	\begin{tikzcd}[labels=description, column sep=1.4cm]
		&&& \xs^0_0
		\ar[dl, bend left, "w|Z",pos=0.7]
		\ar[dr, "z|Z", bend left]
		\ar[ddd,bend left=20, "\substack{s|Z\\ +t|Z}",pos=.7]
		\ar[d,bend left =80, "{zw|Z}"]
		&&
		\\[.5cm]
		\cdots
		\ar[r,bend left, "z|U"]
		& \xs_{-2}
		\ar[r, bend left, "z|U"]
		\ar[l, bend left, "w|1"]
		\ar[dd,"\substack{s|U^2\\ +t|1}"]
		&
		\xs_{-1}
		\ar[ur, bend left, "z|W"]
		\ar[l, "w|1", bend left]
		\ar[dd, "\substack{s|U \\ +t|1 }"]
		&
		\xs_0^1
		\ar[u, "1|W"]
		\ar[d, "1|Z"]
		&
		\xs_1
		\ar[r, bend left, "z|1"]
		\ar[dl, "w|Z", bend left]
		\ar[dd, "\substack{s|1 \\ +t|U}"]
		&
		\xs_2
		\ar[l, bend left, "w|U"]
		\ar[r,bend left, "z|1"]
		\ar[dd, "\substack{s|1\\+t|U^2}"]
		&
		\cdots
		\ar[l,bend left, "w|U"]
		\\[.5cm]
		&&&
		\xs^2_0
		\ar[ul, "w|W", bend left]
		\ar[ur, "z|W", bend left,crossing over,pos = 0.75]
		\ar[d,bend right, "\substack{s|W\\ +t|W}"]
		\ar[u,bend left =70, "{wz|W}"]
		&&
		\\[1.5cm]
		\cdots
		\ar[r, bend left, "\varphi_+|1"]
		&T^{-2}\xs'
		\ar[r, bend left, "\varphi_+|1"]
		\ar[l, bend left, "\varphi_-|1"]
		& T^{-1}\xs'
		\ar[r, bend left, "\varphi_+|1"]
		\ar[l, bend left, "\varphi_-|1"]
		&T^{0}\xs'
		\ar[r, "\varphi_+|1", bend left]
		\ar[l, "\varphi_-|1", bend left]
		&T^{1}\xs'
		\ar[r, bend left, "\varphi_+|1"]
		\ar[l, bend left,"\varphi_-|1"]
		&T^{2}\xs'
		\ar[r, bend left, "\varphi_+|1"]
		\ar[l, bend left,"\varphi_-|1"]
		&\cdots 
		\ar[l, bend left,"\varphi_-|1"]
	\end{tikzcd}
}
\]

Next, we compute the length $1$ map $\Phi^{+K_2} +\Phi^{-K_2}:  {}^{\cR_0^!} \cX(\Wh_+)^{\cR_0} \to  {}^{\cR_0^!} \cX(\Wh_+)^{\cR_1}$ by the description in Section \ref{sec:horizontal map from (0,0) to (0,1)}. The result is drawn in the following picture.
	\[	
\adjustbox{scale=0.9}{
	\begin{tikzcd}[labels=description] 
		\\[1.2cm]
		{}^{\cR_0^!} \cX(\Wh_+)^{\cR_0}\ar[dd, "\Phi^{+K_2} +\Phi^{-K_2} \quad ="]
		\\  [.5cm]
		\\ [2cm]		
		{}^{\cR_0^!} \cX(\Wh_+)^{\cR_1}
	\end{tikzcd}
	\hspace{-.2cm}
	\begin{tikzcd}[labels=description, column sep=1.4cm]
		&&& \xs^0_0
		\ar[dl, bend left, "w|Z",pos=0.7]
		\ar[dr, "z|Z", bend left]
		\ar[ddd,bend left=20, "\substack{1|T\sigma\\ +1|UT\tau}",pos=.7]
		\ar[d,bend left =80, "{zw|Z}"]
		&&
		\\[.5cm]
		\cdots
		\ar[r,bend left, "z|U"]
		& \xs_{-2}
		\ar[r, bend left, "z|U"]
		\ar[l, bend left, "w|1"]
		\ar[dd,"\substack{1|\sigma\\ +1|\tau}"]
		&
		\xs_{-1}
		\ar[ur, bend left, "z|W"]
		\ar[l, "w|1", bend left]
		\ar[dd, "\substack{1|\sigma \\ +1|\tau }"]
		&
		\xs_0^1
		\ar[u, "1|W"]
		\ar[d, "1|Z"]
		&
		\xs_1
		\ar[r, bend left, "z|1"]
		\ar[dl, "w|Z", bend left]
		\ar[dd, "\substack{1|\sigma \\ +1|\tau}"]
		&
		\xs_2
		\ar[l, bend left, "w|U"]
		\ar[r,bend left, "z|1"]
		\ar[dd, "\substack{1|\sigma\\+1|\tau}"]
		&
		\cdots
		\ar[l,bend left, "w|U"]
		\\[.5cm]
		&&&
		\xs^2_0
		\ar[ul, "w|W", bend left]
		\ar[ur, "z|W", bend left,crossing over,pos = 0.75]
		\ar[d,bend right, "\substack{1|UT^{-1}\sigma\\ +1|T^{-1}\tau}"]
		\ar[u,bend left =70, "{wz|W}"]
		&&
		\\[1.5cm]
		\cdots
		\ar[r, bend left, "z|U"]
		&\ys_{-2}
		\ar[r, bend left, "z|U"]
		\ar[l, bend left, "w|1"]
		& \ys_{-1}
		\ar[r, bend left, "z|U"]
		\ar[l, bend left, "w|1"]
		& \ys_{0}
		\ar[r, "z|1", bend left]
		\ar[l, "w|1", bend left]
		& \ys_{1}
		\ar[r, bend left, "z|1"]
		\ar[l, bend left,"w|U"]
		& \ys_{2}
		\ar[r, bend left, "z|1"]
		\ar[l, bend left,"w|U"]
		&\cdots 
		\ar[l, bend left,"w|U"]
	\end{tikzcd}
}
\]

For example, we have 
\[
\gr(R^{\sigma}(\xs^0_0)) = \gr(\xs^0_0)+ \gr(R^{\sigma}) =(0,-2,0,1)+(0,2,0,0) = (0,0,0,1), 
\]
where $\gr =(\gr_{\ws},\gr_{\zs},A_1,A_2)$ and $\ys_0$ has grading $\gr = (0,0,0,0)$. Therefore, 
\[
R^{\sigma}(\xs^0_0) = \ys_0\otimes T,
\]
 where $T$ in the coefficient carries the grading $A_2 =1$. This corresponds to the arrow $1|T\sigma$ from $\xs^0_0$ to $\ys_0$. The rest of the components can be computed in the same manner.

Next, we compute the length $1$ map $\Phi^{+K_1} +\Phi^{-K_1}:  {}^{\cR_0^!} \cX(\Wh_+)^{\cR_1} \to  {}^{\cR_1^!} \cX(\Wh_+)^{\cR_1}$. Since $K_1$ is the unknot, the $0$-framed knot surgery complex ${}_{\cK} \cX_0(K_1)^{\bF[U]}$ has the action
\[L_{\sigma} (\ys) = L_{\tau}(\ys) = \ve{e},\] where $\ys$ (resp. $\ve{e}$) is the unique generator in the $0$ (resp. $1$) idempotent of ${}_{\cK} \cX_0(K_1)^{\bF[U]}$. The linking number of the Whitehead link is $0$. Therefore, by the description in Section \ref{sec:vertical map from (0,1) to (1,1)}, the length $1$ map takes the following form:
\[
\adjustbox{scale=0.8}{
	\begin{tikzcd}[labels=description] 
	{}^{\cR_0^!} \cX(\Wh_+)^{\cR_1}\ar[d, "\Phi^{+K_1} +\Phi^{-K_1} \quad ="]
	\\ [1.5cm]		
	{}^{\cR_1^!} \cX(\Wh_+)^{\cR_1}
\end{tikzcd}
\begin{tikzcd}[labels={description}, column sep = 1.4cm]
		\cdots
		\ar[r, bend left, "z|U"]
		& \ys_{-2}
		\ar[l, bend left, "w|1"]
	\ar[r, bend left, "z|U"]
		\ar[d,"\substack{s|U^2\\ +t|1}"]
	& \ys_{-1}
	\ar[r, bend left, "z|U"]
	\ar[l, bend left, "w|1"]
	\ar[d,"\substack{s|U\\ +t|1}"]
	&\ys_{0}
	\ar[r, "z|1", bend left]
	\ar[l, "w|1", bend left]
	\ar[d,"\substack{s|1\\ +t|1}"]
	&\ys_{1}
	\ar[r, bend left, "z|1"]
	\ar[l, bend left,"w|U"]
	\ar[d,"\substack{s|1\\ +t|U}"]
		&\ys_{2}
	\ar[r, bend left, "z|1"]
	\ar[l, bend left,"w|U"]
	\ar[d,"\substack{s|1\\ +t|U^2}"]
	&\cdots 
	\ar[l, bend left,"w|U"]
	\\[1.5cm]
	\cdots
	\ar[r, bend left, "\varphi_+|1"]
	& T^{-2}\ve{e}
	\ar[r, bend left, "\varphi_+|1"]
	\ar[l, bend left, "\varphi_-|1"]
	& T^{-1}\ve{e}
	\ar[r, bend left, "\varphi_+|1"]
	\ar[l, bend left, "\varphi_-|1"]
	&T^{0}\ve{e}
	\ar[r, "\varphi_+|1", bend left]
	\ar[l, "\varphi_-|1", bend left]
	&T^{1}\ve{e}
	\ar[r, bend left, "\varphi_+|1"]
	\ar[l, bend left,"\varphi_-|1"]
	& T^{2}\ve{e}
	\ar[r, bend left, "\varphi_+|1"]
	\ar[l, bend left, "\varphi_-|1"]
	&\cdots 
	\ar[l, bend left,"\varphi_-|1"]
\end{tikzcd}
}
\]

Finally, we compute the length $1$ map $\Phi^{+K_2} +\Phi^{-K_2}:  {}^{\cR_1^!} \cX(\Wh_+)^{\cR_0} \to  {}^{\cR_1^!} \cX(\Wh_+)^{\cR_1}$. Since $K_2$ is the unknot, the $0$-framed surgery complex $\cX_0(K_2)^{\cK}$ takes the form:
\begin{equation}
	\xs' \xrightarrow{\sigma +\tau} \ve{e},
	\label{eq:surgery complex of the unknot}
\end{equation}
 where $\xs'$ (resp. $\ve{e}$) is the unique generator in the $0$ (resp. $1$) idempotent of $\cX_0(K_2)^{\cK}$. Then, by the description in Section \ref{sec:horizontal map from (1,0) to (1,1)}, the length 1 map takes the following form

\[
\adjustbox{scale=0.8}{
	\begin{tikzcd}[labels=description] 
		{}^{\cR_1^!} \cX(\Wh_+)^{\cR_0}\ar[d, "\Phi^{+K_2} +\Phi^{-K_2} \quad ="]
		\\ [1.5cm]		
		{}^{\cR_1^!} \cX(\Wh_+)^{\cR_1}
	\end{tikzcd}
	\begin{tikzcd}[labels={description}, column sep = 1.4cm]
		\cdots
	\ar[r, bend left, "\varphi_+|1"]
	&T^{-2}\xs'
	\ar[r, bend left, "\varphi_+|1"]
	\ar[l, bend left, "\varphi_-|1"]
	\ar[d, "\substack{1|\sigma\\+1|\tau}"]
	& T^{-1}\xs'
	\ar[r, bend left, "\varphi_+|1"]
	\ar[l, bend left, "\varphi_-|1"]
	\ar[d, "\substack{1|\sigma\\+1|\tau}"]
	&T^{0}\xs'
	\ar[r, "\varphi_+|1", bend left]
	\ar[l, "\varphi_-|1", bend left]
	\ar[d, "\substack{1|\sigma\\+1|\tau}"]
	&T^{1}\xs'
	\ar[r, bend left, "\varphi_+|1"]
	\ar[l, bend left,"\varphi_-|1"]
	\ar[d, "\substack{1|\sigma\\+1|\tau}"]
	&T^{2}\xs'
	\ar[r, bend left, "\varphi_+|1"]
	\ar[l, bend left,"\varphi_-|1"]
	\ar[d, "\substack{1|\sigma\\+1|\tau}"]
	&\cdots 
	\ar[l, bend left,"\varphi_-|1"]
		\\[1.5cm]
		\cdots
		\ar[r, bend left, "\varphi_+|1"]
		& T^{-2}\ve{e}
		\ar[r, bend left, "\varphi_+|1"]
		\ar[l, bend left, "\varphi_-|1"]
		& T^{-1}\ve{e}
		\ar[r, bend left, "\varphi_+|1"]
		\ar[l, bend left, "\varphi_-|1"]
		&T^{0}\ve{e}
		\ar[r, "\varphi_+|1", bend left]
		\ar[l, "\varphi_-|1", bend left]
		&T^{1}\ve{e}
		\ar[r, bend left, "\varphi_+|1"]
		\ar[l, bend left,"\varphi_-|1"]
		& T^{2}\ve{e}
		\ar[r, bend left, "\varphi_+|1"]
		\ar[l, bend left, "\varphi_-|1"]
		&\cdots 
		\ar[l, bend left,"\varphi_-|1"]
	\end{tikzcd}
}
\]

The length $2$ maps $\Phi^{\pm K_1,\pm K_2}: {}^{\cR_0^!} \cX(\Wh_+)^{\cR_0} \to  {}^{\cR_1^!} \cX(\Wh_+)^{\cR_1}$ vanish.

\subsection{The Whitehead knot in $S^1\times S^2$}

We now compute the surgery complex of the Whitehead knot $W\subset S^1\times S^2$ obtained by performing 0-surgery to $K_1$. Using the type-AA bimodule ${}_{\cK}[\cTr]_{\cK^!}$ defined in \cite{ZemKoszul}*{Section~3.4}, we have 
\begin{equation}
\cX_{0}(S^1\times S^2,W) ^{\cK} \simeq \cX_{0}(U)^{\cK}\boxtimes {}_{\cK}[\cTr]_{\cK^!} \boxtimes {}^{\cK^{!}} \cX^{\alg}(\Wh_+)^{\cK}, \label{eq:Whitehead-knot-tensor-product}
\end{equation} 
Here, $U$ is the unknot in $S^3$, with the $0$-framed surgery complex $\cX_{0}(U)^{\cK}$ shown below:
\[
\cX_0(U)^{\cK}=(\begin{tikzcd} \ve{X}_0 \ar[r, "\sigma+\tau"] & \ve{X}_1 \end{tikzcd}).
\]
%
%
%
% and   ${}_{\cK}[\cTr]_{\cK^!}$ is an $AA$-bimodule whose underlying $(\ve{I},\ve{I})$-bimodule is $\ve{I} = \bF\langle i_0,i_1 \rangle$, with the bimodule action defined using the homological pertubation lemma. See \cite{ZemKoszul}*{Lemma~3.9} for a more detailed description of this $AA$-bimodule.
 
 It is straightforward to see directly from the construction in \cite{ZemKoszul}*{Section~3.4} that the only sequence of monomials $a_1,\dots, a_n\in \cK^!$ such that
 \[
 m_{1|1|n}(\sigma,1,a_1,\dots, a_n)\neq 0
 \]
 has $n=1$ and $a_1=s$. The analogous claim holds for sequences that pair non-trivially with $\tau$. 
 
 The idempotent 0 subspace of $\cX_0(S^1\times S^2,W)^{\cK}$ is therefore given by the infinite complex
 \[
 	\begin{tikzcd}[ column sep=1cm, row sep=.5cm]
 		&& \ve{X}_0|\xs^0_0
 		&&
 		\\[.5cm]
 		\cdots
 		 \ve{X}_0|\xs_{-2}
 		\ar[dd,"1+U^2"]
 		&
 		\ve{X}_0|\xs_{-1}
 		\ar[dd, "1+U"]
 		&
 		\ve{X}_0|\xs_0^1
 		\ar[u, "W"]
 		\ar[d, "Z"]
 		&
 		\ve{X}_0|\xs_1
 		\ar[dd, "1+U"]
 		&
 		\ve{X}_0|\xs_2
 		\ar[dd, "1+U^2"]
 		\cdots
 		\\[.5cm]
 		&&
 		\ve{X}_0|\xs^2_0
 		&&
 		\\
 		\cdots
 		\ve{X}_1|T^{-2}\xs'
 		& \ve{X}_1|T^{-1}\xs'
 		&\ve{X}_1|T^{0}\xs'
 		&\ve{X}_1|T^{1}\xs'
 		&\ve{X}_1|T^{2}\xs'
 		\cdots 
 	\end{tikzcd}
 \]
Idempotent $1$ of the tensor product is given by
\[
\begin{tikzcd}[ column sep = .8cm, row sep=.5cm]
		\cdots \ve{X}_0|\ys_{-2}
		\ar[d,"1+U^2"]
	& \ve{X}_0|\ys_{-1}
	\ar[d,"1+U"]
	&\ve{X}_0|\ys_{0}
	&\ve{X}_0|\ys_{1}
	\ar[d,"1+U"]
		&\ve{X}_0|\ys_{2}
	\ar[d,"1+U^2"]\cdots 
	\\[1.5cm]
	\cdots \ve{X}_1|T^{-2}\ve{e}
	& \ve{X}_1|T^{-1}\ve{e}
	&\ve{X}_1|T^{0}\ve{e}
	&\ve{X}_1|T^{1}\ve{e}
	& \ve{X}_1|T^{2}\ve{e}
	\cdots 
\end{tikzcd}
\]
We note that in both idempotents $(0,0)$ and $(1,1)$ of $\cK$, the element $1+U^n$ is a unit if $n>0$ since we complete $\cK$ with respect to the $U$-adic topology (see \cite{ZemExact}*{Section~4.2} for more on completions on $\cK$). In particular, after homotopy equivalence, we may remove all of the generators except
\[
\ve{X}_0|\xs_0^0,\quad \ve{X}_0|\xs_0^1, \quad\ve{X}_0|\xs_0^2,\quad \ve{X}_1|T^0 \xs',\quad \ve{X}_0|\ve{y}_0, \quad \text{and} \quad \ve{X}_1|T^0 \ve{e}.
\] 
Furthermore, using the description of $\Phi^{K_2}$ and $\Phi^{-K_2}$ above, we compute easily that the resulting model for $\cX_{0}(S^1\times S^2, W)^{\cK}$ takes the following form:
\[
	\cX_0( S^1\times S^2,W)^{\cK}\,\,\,\,\simeq \quad 
	\begin{tikzcd}[ row sep = 1.2cm]
		& \ve{X}_0|\xs^1_0
			 \arrow[dl,swap, "W"] 
			 \arrow[dr, "Z"]&&[-1cm] &[-.7cm] \\
		\ve{X}_0|\xs^0_0
			 \arrow[dr,swap, "UT^{-1}\sigma+T^{-1}\tau"]& &  \ve{X}_0|\xs^2_0 \arrow[dl,"T\sigma + UT\tau"]  & |[yshift=-30pt]| \oplus&  \ve{X}_1|  T^0\xs' \arrow[d, "\sigma+\tau"]\\[-0.7cm]
		& \ve{X}_0| \ys_0 & & & \ve{X}_1| T^0\ve{e}
	\end{tikzcd}
\]
\begin{rem} Note that $\cX_0(S^1\times S^2,W)^{\cK}$ is isomorphic to $\cX_0(T_{2,3})^{\cK}\oplus \cX_0(U)^{\cK}$.
\end{rem}

\bibliographystyle{custom}
\def\MR#1{}
\bibliography{biblio}

\end{document}